\documentclass[9pt]{amsart}
\usepackage{amsthm, amsfonts, amssymb, amsmath, enumitem, bbold}
\usepackage[all]{xy}
\usepackage{avant}

\newcommand{\C}{\mathbb{C}}

\newcommand{\Z}{\mathbb{Z}}

\newcommand{\F}{\mathbb{F}}

\newcommand{\e}{\epsilon}

\newcommand{\ct}{\mathbb{1}}
\newtheorem{thm}{Theorem}[section]
\newtheorem{lem}[thm]{Lemma}
\newtheorem{cor}[thm]{Corollary}
\newtheorem{prop}[thm]{Proposition}
\newtheorem*{theorem*}{Theorem}

\theoremstyle{definition}

\newtheorem*{ex}{Example}
\newtheorem*{defn}{Definition}

\newtheorem*{notes}{Notes}
\newtheorem*{datum}{Datum}

\newtheorem*{ack}{Acknowledgments}

\begin{document}
\title{Kloosterman sums over finite Frobenius rings}

\date{\today}
\author{Bogdan Nica}
\begin{abstract}
We study Kloosterman sums in a generalized ring-theoretic context, that of finite commutative Frobenius rings. We prove a number of identities for twisted Kloosterman sums, loosely clustered around moment computations. 
\end{abstract}

\address{\newline Department of Mathematics and Statistics \newline McGill University, Montreal}
\maketitle

\section*{Introduction}
A classical Kloosterman sum is an exponential sum of the form
\begin{align*}
S(m,n; q)=\sum_{u\in (\Z/q\Z)^\times} e_q(mu+nu^{-1})
\end{align*}
where $m,n,q$ are positive integers, $(\Z/q\Z)^\times$ denotes the unit group of the modular ring $\Z/q\Z$, and $e_q$ is the $q$-scaled exponential function $r\mapsto \exp(2\pi i r/q)$. These sums have played an important role in analytic number theory, starting with Kloosterman \cite{K26} up until recent developments, see \cite{M06} for instance. Two excellent references that highlight, among other things, the classical Kloosterman sums and their use in analytic number theory, are \cite{I} and \cite{IK}.

The formulation of Kloosterman sums is, however, distinctly algebraic: they are character sums that make sense over an arbitrary finite ring. Throughout this paper, rings have a multiplicative identity, and are assumed to be commutative--though, certainly, Kloosterman sums can also be defined over non-commutative rings such as matrix rings over finite commutative rings. 

Let $R$ be a finite ring.  The unit group of $R$, that is to say, the multiplicative group of invertible elements of $R$, is denoted by $R^\times$. A \emph{Kloosterman sum} over $R$ is a character sum of the form
\[K(\phi,\psi)=\sum_{u\in R^\times} \phi(u)\:\psi(u^{-1})\]
where $\phi$ and $\psi$ are additive characters of $R$.

If $R=\Z/q\Z$ then every additive character is of the form $r\mapsto e_q(mr)$, and one recovers the classical Kloosterman sum in this case. Kloosterman sums have been considered over many other finite rings: finite fields, quotient rings of rings of integers \cite{P98, P00, E01}, Galois rings \cite{SKH}. The starting aim of this paper is that of establishing some fundamental properties for Kloosterman sums over finite rings. Actually, a certain class of finite rings emerges naturally as the broadest in which a meaningful theory of classical character sums can be carried out--the class of finite Frobenius rings. For Gauss and Jacobi sums, the general viewpoint of Frobenius rings has already been explored \cite{L, W, Sz}. (It may be interesting to note, however, that neither of these references uses the Frobenius terminology.) A similar investigation of Kloosterman sums is entirely new, to the best of our knowledge, and the present paper is a first step in this direction. Part of the motivation in seeking generality is that the usefulness of Kloosterman sums, like that of many other character sums, has long transcended number theory and the finite rings that appear naturally therein. See, for instance, \cite{H} for applications of Kloosterman sums to coding theory. Incidentally, finite Frobenius rings are also the natural context for algebraic coding theory \cite{Woo, Woo2, D, SAS}.

Basic facts on Kloosterman sums over finite rings are presented in Section~\ref{sec: kloos}. Preliminaries on additive and multiplicative characters, highlighting the concept of primitive characters (Sections~\ref{sec: addchar} and ~\ref{sec: multichar}) prepare the ground for Section~\ref{sec: Frob}, where we discuss finite Frobenius rings in some detail.

Our main results start in Section~\ref{sec: SK}. In Theorem~\ref{thm: gSK} we extend the Selberg-Kuznetsov identity, satisfied by the classical Kloosterman sums over the modular ring $\Z/q\Z$, to Kloosterman sums over finite Frobenius rings. We point out that a generalization of the Selberg-Kuznetsov identity to quotients of algebraic rings of integers is given in \cite{P98}. Our ring-theoretic context is far broader. 

Starting with Section~\ref{sec: twistedK}, we somewhat change our perspective. The essential Kloosterman sums attached to a finite Frobenius ring $R$ come as a parameterized family $\{K(a): a\in R\}$, where
\[K(a)=\sum_{u\in R^\times} \psi(u+au^{-1}).\] 
The parameterization is relative to a choice of a primitive additive character $\psi$ of $R$. This parameterized viewpoint is generalized by the \emph{twisted Kloosterman sums}, a family $\{K_\tau(a): a\in R\}$ of sums which involve, in addition, a fixed multiplicative character $\tau$ of $R$: 
\[K_\tau(a)=\sum_{u\in R^\times} \tau(u)\:\psi(u+au^{-1}).\]
Note that the usual Kloosterman family corresponds to the trivial twist. Such twisted Kloosterman sums were first considered, over a finite field $\Z/p\Z$, by Davenport \cite{Dav}. Kloosterman himself studied them, in some later work \cite{K46}, over the modular ring $\Z/q\Z$ with $q$ a power of a prime. Twisted Kloosterman sums over modular rings $\Z/q\Z$, where $q$ is a power of an odd prime, in which the twist $\tau$ is the quadratic character, are known as \emph{Sali\'e sums}. They first appeared in an early, important paper of Sali\'e \cite{S}. Twisted Kloosterman sums over finite Frobenius rings are introduced, together with their basic properties, in Section~\ref{sec: twistedK}.

The study of Kloosterman sums, twisted or not, can be quickly reduced to the case of local Frobenius rings. In this case, one can then distinguish between primitive and non-primitive multiplicative characters. Our most substantial results address twisted Kloosterman sums by a non-primitive twist $\tau$. This case includes, notably, the usual Kloosterman sums, as well as the Sali\'e sums in odd characteristic.

\begin{theorem*}
Assume $R$ is a local Frobenius ring, but not a field. 

For any twist $\tau$, the following hold:
\begin{align*}
\sum_{a\in R} K_\tau(a)=\sum_{a\in R^\times} K_\tau(a)=0,\qquad \sum_{a\in R} \big|K_\tau(a)\big|^2=|R||R^\times|,\\\sum_{a\in R^\times} \big|K_\tau(a)\big|^2=\begin{cases} |R||R^\times|& \textrm{ if } \tau  \textrm{ is non-primitive,}\\
2|R||R^\times|-|R|^2 &  \textrm{ if } \tau  \textrm{ is primitive.}
\end{cases}
\end{align*}

For any non-primitive twist $\tau$, the following hold.
\begin{itemize}
 \item[\textsc{(i)}] We have $K_\tau(a)=0$ whenever $a\notin R^\times$.  If $R$ has odd characteristic then, moreover, $K_\tau(a)=0$ whenever $a\notin (R^\times)^2$. 
 \item[\textsc{(ii)}] If $R$ has odd characteristic, then
 \begin{align*}
 \sum_{a\in R^\times} K_\tau(a)^3=0, \qquad \sum_{a\in R^\times} \big|K_\tau(a)\big|^4=3|R^\times||R|^2.
 \end{align*}
 \end{itemize}
\end{theorem*}

The first part of the theorem collects moment computations that are valid for arbitrary twists. These are established in Proposition~\ref{prop: tK1}, Corollary~\ref{cor: 1sttwist}, and Corollary~\ref{cor: 2ndtwist}.

The second part deals with non-primitive twists. Item \textsc{(i)} is shown in Theorem~\ref{thm: tK2}. Item \textsc{(ii)} is covered in Theorem~\ref{thm: 3mom}, whose hypothesis is actually a bit stronger, respectively Theorem~\ref{thm:4mom}.

As we have already mentioned, the first aim of the paper is that of studying Kloosterman sums over finite rings. This includes a number of basic properties, but also the identification of the natural ring-theoretic habitat for these sums--namely, finite Frobenius rings. The second aim of the paper is that of explicitly computing the first few moments of the twisted Kloosterman sums over finite Frobenius rings. Perhaps the most interesting such computation is that of the fourth moment:
\[\sum_{a\in R^\times} \big|K_\tau(a)\big|^4=3|R^\times||R|^2\]
whenever $R$ is a local Frobenius ring of odd characteristic, but not a field, and the twist $\tau$ is non-primitive. For one thing, it is interesting because of the technical aspect. We achieve it by adopting what is essentially a Fourier-theoretic perspective. We compute certain multiplicative Fourier transforms, herein framed as weighted moments of twisted Kloosterman sums. These weighted moments link twisted Kloosterman sums with Gauss and Jacobi sums.

Another reason for the interest is historical. Originally, the fourth moment was used to derive non-trivial upper bounds on Kloosterman sums. The trivial bound is
\[\max_{a\in R^\times} |K(a)|\ll|R|,\] 
and by a non-trivial bound we mean a power-saving bound of the form $|R|^\e$, where $\e<1$. (Throughout this paragraph, the implied constants are absolute.) Kloosterman's idea \cite{K26} was to estimate the fourth moment in order to derive the non-trivial bound
 \[\max_{a\in R^\times} |K(a)|\ll|R|^{3/4}\] 
in the prime field case $R=\Z/p\Z$.  Weil \cite{Weil}, see also \cite[Ch.11]{IK}, famously improved the bound to $|R|^{1/2}$ in the case when $R$ is a field. It is not hard to see, for instance from a second moment computation, that the square root bound is best possible. Sali\'e \cite{S} (see also Estermann \cite{E}) has generalized and improved Kloosterman's bound in the case $R=\Z/q\Z$, where $q$ is a proper power of a prime. More recently, Knightly and Li \cite[Thm.9.3] {KL} have proved similar bounds for twisted Kloosterman sums in the same case, $R=\Z/q\Z$ with $q$ a proper power of a prime.

The fourth Kloosterman moment over a finite field of odd characteristic can be computed explicitly, and is given by 
\[\sum_{a\in R^\times} K(a)^4=2|R^\times||R|^2-|R|^2-3|R|-1.\]
This is due to Sali\'e \cite{S} in the prime field case $R=\Z/p\Z$. His method works in general, however. See also \cite[Sec.4.4]{I}, though note that formula (4.25) therein needs a small correction. The fourth moment computation in the field case is done by solution-counting, a method which does not seem to be applicable in the present ring-theoretic generality even if we merely want the fourth moment of the usual Kloosterman sums.

The fourth moment of twisted Kloosterman sums over $R=\Z/q\Z$, where $q$ is a proper power of a prime, was investigated by W. Zhang in several papers \cite{Z04, Z, Z16, ZS}. Zhang and Shen \cite{ZS} deal with the case of a non-primitive twist, and our computation is a broad generalization of their results. In \cite{Z16}, Zhang treats the case of a primitive twist. We hope to address this case, in the generality of local Frobenius rings, in a future work. 

\begin{ack}
I would like to thank Youri Tamitegama for interesting discussions around Kloosterman sums. 
\end{ack}

\section{Preliminaries}\label{sec: alg}

\subsection{Rings} Let $R$ be a finite ring. Recall that all rings considered in this paper are commutative, and have an multiplicative identity.

If $I$ is an ideal of $R$, then we write $[\:\cdot\:]: R\to R/I$ for the quotient ring homomorphism. In order to avoid the degenerate zero ring, we usually quotient by $I\neq R$. 

If $I$ is an ideal of $R$, then the \emph{annihilator} of $I$ is the ideal
\begin{align*}
I^\perp=\{r\in R: ra=0 \textrm{ for all } a\in I\}.
\end{align*}

A finite ring decomposes into a product of finite local rings. In contexts with good product behaviour, we can then focus on finite local rings--a very helpful simplification.

Let $R$ be a finite local ring, with maximal ideal $M$. Then the unit group $R^\times$ can be described as the complement of $M$ in $R$. The quotient $R/M$ is the \emph{residue field} of $R$. Note that $1+M$ is a subgroup of $R^\times$. More generally, let $I\neq R$ be an ideal of $R$. Then $R/I$ is again a finite local ring, and $u$ is a unit in $R$ if and only if $[u]$ is a unit in $R/I$. In particular, the induced group homomorphism $[\:\cdot\:]: R^\times\to (R/I)^\times$ is surjective, and $1+I$ is a subgroup of $R^\times$.

\subsection{Characters} Let $G$ be a finite abelian group. Recall that a \emph{character} of $G$ is a group homomorphism $G\to \C^*$. The \emph{dual} $\widehat{G}$ is the set of characters of $G$. Then $\widehat{G}$ is an abelian group under pointwise multiplication, the identity element being the trivial character $\ct$. Furthermore, $\widehat{G}$ is (non-canonically) isomorphic to $G$. In particular the dual group has the same size as the original group: $|\widehat{G}|=|G|$.

Character sum computations are carefully orchestrated applications of the basic orthogonality relations: 
\[\sum_{g\in G} \kappa(g)=|G|\:[\kappa=\ct], \qquad \sum_{\kappa\in \widehat{G}} \kappa(g)=|G|\:[g=0]\]
for each character $\kappa\in \widehat{G}$, respectively for each element $g\in G$. Here, and several times throughout this paper, we find it convenient to use the Iverson bracket notation. This is a propositional variation on the Kronecker delta notation, defined as follows: if $\mathcal{P}$ is a proposition, then
\begin{align*}
[\mathcal{P}]=\begin{cases} 
1 & \textrm{if } \mathcal{P} \textrm{ is true,}\\
0 & \textrm{if } \mathcal{P} \textrm{ is false.}
\end{cases}
\end{align*}

\section{Additive characters}\label{sec: addchar}
By a \emph{character} of a finite ring $R$ we always mean a character of the additive group of $R$. Additive characters take, naturally, precedence herein, as they underlie the Kloosterman sums. But in the tapestry of character sums, multiplicative characters are inevitably interwoven as well. They will appear later on, starting with Section~\ref{sec: multichar}, and they will be explicitly labeled as such.

\subsection{Primitive characters}
Let $\psi$ be a character of $R$. Note that any additive subgroup of $R$ contains a largest ideal. Applied to the kernel of $\psi$, this fact says that there is a largest ideal $I$ so that $\psi=\ct$ on $I$; the ideal $I$ is called the \emph{conductor} of $\psi$. The trivial character $\ct$ has conductor $R$. At the other end of the conductor spectrum, we have the most non-trivial characters.

\begin{defn}
A character $\psi$ is said to be \emph{primitive} if its conductor is the zero ideal.
\end{defn}

Here are two simple, but fundamental examples. Consider $R=\Z/q\Z$. Then the $q$-scaled exponential $e_q(r)=\exp(2\pi i r/q)$ is a primitive character, since the kernel of $e_q$ is $(0)$. Consider now $R=\F_q$, a finite field with $q$ elements. Let $\psi$ be a non-trivial character. Then $\psi$ is automatically primitive, as $(0)$ is the only ideal other than $R$. Note, however, that the kernel of $\psi$ is not reduced to $(0)$ as soon as $q$ is a proper prime power. 

The characters of $R$ which are trivial on an ideal $I$ are in a natural correspondence to the characters of the quotient ring $R/I$.  Characters of $R$ with conductor $I$ correspond to primitive characters of $R/I$.

\subsection{Scaling} Let $\psi$ be a character of $R$, and let $a\in R$. Then the scaling $a.\psi$, given by $r\mapsto \psi(ar)$, is again a character of $R$. 

\begin{lem}\label{lem: scale}
Let $\psi$ be a primitive character. Then the following hold:
\begin{itemize}
\item[(i)] each character of $R$ is of the form $a.\psi$ for a unique $a\in R$;
 \item[(ii)] $a.\psi$ is primitive if and only if $a\in R^\times$. 
 \end{itemize}
\end{lem}

\begin{proof} (i) A character $\psi$ induces a homomorphism of abelian groups $h_\psi: R\to \widehat{R}$, given by $a\mapsto a.\psi$. The injectivity of $h_\psi$ is equivalent to $\psi$ being primitive--our very hypothesis. As $|R|=|\widehat{R}|$, it follows that $h_\psi$ is a bijection--the desired claim.

(ii) The conductor of $a.\psi$ is the annihilator ideal $(a)^\perp=\{r\in R: ra=0\}$. Thus $a.\psi$ is primitive if and only if $(a)^\perp=(0)$, that is to say, $a$ is not a zero-divisor. As $R$ is finite, this is equivalent to $a$ being a unit.
\end{proof}

The proof of part (i) has a bit more to say. On the one hand, the \emph{generating} property for $\psi$ captured by (i) is in fact equivalent to $\psi$ being primitive. On the other hand, note that the additive dual $\widehat{R}$ is an $R$-module under the scaling action, and the map $h_\psi$ induced by a character $\psi$ is an $R$-module homomorphism. So, a primitive character $\psi$ induces an $R$-module isomorphism between $R$ and $\widehat{R}$. We will return to this idea in Section~\ref{sec: Frob}.

In the context of character sums, the following summation formula will prove useful.

\begin{lem}\label{lem: primsum}
 Let $\psi$ be a primitive character. If $I$ is an ideal of $R$, then
\begin{align*}
\sum_{a\in I} \psi(ca)=|I| \: [c\in I^\perp].
\end{align*}
\end{lem}

\begin{proof}
Indeed, $\sum_{a\in I} \psi(ca)= \sum_{a\in I} (c.\psi)(a)$ equals $|I|$ or $0$ according to whether the scaled character $c.\psi$ is, or is not, trivial on the ideal $I$. As $\psi$ is primitive, $c.\psi=\ct$ on $I$ if and only if $cI=(0)$, that is $c\in I^\perp$.
\end{proof}

\subsection{Products}
Assume $R$ is a product ring, $R=R_1\times\dots \times R_n$. If $\psi_k$ is a character on $R_k$, for each $k$, then $\psi_1\times\dots \times \psi_n$, given by $(r_1, \dots, r_n)\mapsto \psi_1(r_1)\dots \psi_n(r_n)$, is a character on $R$. Furthermore, each character $\psi$ on $R$ is of the above form, in a unique way. So there is an isomorphism $\widehat{R_1}\times\dots \times \widehat{R_n}\simeq \widehat{R}$, given by $(\psi_1,\dots,\psi_n)\mapsto \psi_1\times\dots \times \psi_n$, which matches them as abelian groups, and even as $R$-modules.

The following fact is expected, and easy to prove.

\begin{lem}\label{lem: condprod}
Consider the character $\psi_1\times\dots \times \psi_n$ on $R$, where each $\psi_k$ is a character on $R_k$. Then the conductor of $\psi_1\times\dots \times \psi_n$ is $I_1\times\dots\times I_n$, where $I_k$ is the conductor of $\psi_k$. In particular, $\psi_1\times\dots \times \psi_n$ is primitive if and only if each $\psi_k$ is primitive.
\end{lem}


\section{Kloosterman sums over finite rings}\label{sec: kloos}
We now turn to Kloosterman sums over finite rings, and establish a few general properties. Let us begin by recalling the following.

\begin{defn}
Let $R$ be a finite ring. A \emph{Kloosterman sum} over $R$ is a character sum of the form
\[K(\phi,\psi)=\sum_{u\in R^\times} \phi(u)\:\psi(u^{-1})\]
where $\phi$ and $\psi$ are characters of $R$. 
\end{defn}

The two facts below are straightforward, and the proofs are left to the reader.

\begin{prop}
Let $\phi$ and $\psi$ be characters of $R$. Then the following hold:
\begin{itemize}
\item[(i)] (Reality) $K(\phi,\psi)$ is a real number.
\item[(ii)] (Symmetry) $K(\phi,\psi)=K(\psi,\phi)$.
\item[(iii)] (Unit-shifting) If $a\in R^\times$, then $K(\phi,a.\psi)=K(a.\phi,\psi)$.
\end{itemize}
\end{prop}

\begin{prop}[Factorization]\label{lem: fac} 
Assume $R$ is a product ring, $R=R_1\times\dots \times R_n$, and let $\phi=\phi_1\times\dots \times \phi_n$ and $\psi=\psi_1\times\dots \times \psi_n$ be characters of $R$. Then
\begin{align*}
K(\phi,\psi)=K(\phi_1,\psi_1)\dots K(\phi_n,\psi_n)
\end{align*}
where each factor $K(\phi_k,\psi_k)$ is a Kloosterman sum over $R_k$.
\end{prop}

The conceptual upshot of the factorization property is the following: since a finite ring decomposes into a product of finite local rings, and Kloosterman sums factor correspondingly, we can focus on Kloosterman sums for finite local rings. 

Firstly, let us record the values of those Kloosterman sums over local rings which involve the trivial character. Such sums are also known as Ramanujan sums.

\begin{prop}\label{prop: rama}
Assume $R$ is a local ring, with maximal ideal $M$. Let $\psi$ be a character of $R$. Then:
\begin{align*}
K(\psi, \ct)=\begin{cases} |R|\: [\psi=\ct]-|M| & \textrm{ if } \psi \textrm{ is trivial on } M,\\ 0 & \textrm{ otherwise}. \end{cases}
\end{align*}
In particular, if $\psi$ is primitive then 
\begin{align*}
K(\psi, \ct)=-\:[ R \textrm{ is a field}].
\end{align*}
\end{prop}

\begin{proof} Write
\[K(\psi, \ct)=\sum_{u\in R^\times} \psi(u)=\sum_{a\in R} \psi(a)-\sum_{a\in M} \psi(a).\]
The right-hand side equals $|R|-|M|=|R^\times|$ if $\psi$ is trivial; $-|M|$ if $\psi$ is non-trivial, but trivial on $M$; $0$ if $\psi$ is non-trivial on $M$. When $\psi$ is primitive, the outcome hinges on whether its conductor, the zero ideal, equals $M$ or not. \end{proof}

\begin{prop}\label{prop: van-red} 
Assume $R$ is a local ring. Let $\phi$ and $\psi$ be non-trivial characters of $R$. 
\begin{itemize}
\item[(i)] (Vanishing) If $\phi$ and $\psi$ have different conductors, then $K(\phi,\psi)=0$.
\item[(ii)] (Reduction) If $\phi$ and $\psi$ have conductor $I$, then 
\[K(\phi,\psi)=|I|\: K(\phi',\psi')\]
where $\phi'$ and $\psi'$ are the primitive characters on $R/I$ inducing $\phi$ and $\psi$, and $K(\phi',\psi')$ is a Kloosterman sum over $R/I$.
\end{itemize}
\end{prop}

\begin{proof}
(i) Let $I(\phi)$ and $I(\psi)$ be the conductors of $\phi$, respectively $\psi$. Without loss of generality, we may assume that $I(\psi)\nsubseteq I(\phi)$. Let $\psi'$ be the additive character on $R/I(\psi)$ which induces $\psi$. We then have
\begin{align*}
K(\phi,\psi)=\sum_{u\in R^\times} \phi(u)\:\psi'([u]^{-1})=\sum_{v\in (R/I(\psi))^\times}\psi'(v^{-1}) \sum_{u\in R^\times,\: [u]=v}\phi(u).
\end{align*}
We now show that each inner sum vanishes. Fix $v\in (R/I(\psi))^\times$, and let $\tilde{u}\in R$ be a lift of $v$. Note that $\tilde{u}\in R^\times$. The full lift $\{u\in R: [u]=v\}=\tilde{u}+I(\psi)$ is also a subset of the unit group $R^\times$, and so
\begin{align*}
\sum_{u\in R^\times,\: [u]=v}\phi(u)=\sum_{a\in I(\psi)}\phi(\tilde{u}+a)=\phi(\tilde{u})\sum_{a\in I(\psi)}\phi(a)=0
\end{align*}
since $\phi$ is non-trivial on $I(\psi)$.

(ii) Let $I$ be the conductor of $\phi$ and $\psi$, and let $\phi'$ and $\psi'$ be the primitive characters on $R/I$ which induce $\phi$ and $\psi$. As $u\in R^\times$ if and only if $[u]\in (R/I)^\times$, we have
\begin{align*}
K(\phi,\psi)=\sum_{u\in R^\times} \phi'([u])\:\psi'([u]^{-1})=|I|\sum_{v\in (R/I)^\times} \phi'(v)\:\psi'(v^{-1}) =|I|\: K(\phi',\psi')
\end{align*}
as claimed.
\end{proof}

Local facts can then be de-localized, using the factorization property of Kloosterman sums. We illustrate this idea by giving general versions of the vanishing and the reduction properties. Given a finite ring $R$, let $R=R_1\times\dots \times R_n$ be a decomposition as a product of local rings. A character $\psi$ of $R$ has a corresponding decomposition $\psi=\psi_1\times\dots \times \psi_n$; each $\psi_k$ is referred to as a local factor of $\psi$.

\begin{cor}
Let $R$ be a finite ring. Let $\phi$ and $\psi$ be characters of $R$, all of whose local factors are non-trivial. 
\begin{itemize}
\item[(i)] (Vanishing) If $\phi$ and $\psi$ have different conductors, then $K(\phi,\psi)=0$.
\item[(ii)] (Reduction) If $\phi$ and $\psi$ have conductor $I$, then $K(\phi,\psi)=|I|\: K(\phi',\psi')$, where $\phi'$ and $\psi'$ are the primitive characters on $R/I$ inducing $\phi$ and $\psi$, and $K(\phi',\psi')$ is a Kloosterman sum over $R/I$.
\end{itemize}
\end{cor}

\begin{proof} Let $\phi=\phi_1\times\dots\times \phi_n$ and $\psi=\psi_1\times\dots\times \psi_n$ be the decompositions into local factors, each of which is non-trivial. 

(i) Since
\[K(\phi,\psi)=\prod_{k=1}^n K(\phi_k,\psi_k),\]
it suffices to argue that one of the factors vanishes. By Lemma~\ref{lem: condprod}, the conductors of $\phi$ and $\psi$ decompose as products of conductors of $\phi_k$'s, respectively $\psi_k$'s. So there is an index $k$ for which 
$\phi_k$ and $\psi_k$ have different conductors, and then $K(\phi_k,\psi_k)=0$ by the local version of vanishing.

(ii) By Lemma~\ref{lem: condprod}, we have $I=I_1\times\dots\times I_n$, where $I_k$ is the conductor of both $\phi_k$ and $\psi_k$. The factorization property, and the local version of reduction, give
\begin{align*}
K(\phi,\psi)=\prod_{k=1}^n K(\phi_k,\psi_k)=\prod_{k=1}^n |I_k|\:K(\phi'_k,\psi'_k)=|I|\:\prod_{k=1}^n K(\phi'_k,\psi'_k).
\end{align*}
Now $R/I$ is naturally isomorphic to $(R_1/I_1)\times\dots \times (R_n/I_n)$. Under this isomorphism, $\phi'$ corresponds to $\phi'_1\times\dots\times \phi'_n$, and $\psi'$ to $\psi'_1\times\dots\times \psi'_n$. Using the factorization property once again, we have
\[\prod_{k=1}^n K(\phi'_k,\psi'_k)=K(\phi',\psi')\]
and the desired claim follows.
\end{proof}

Note that vanishing and reduction need not hold as stated if we merely assume that $\phi$ and $\psi$ are non-trivial characters of $R$.

We end this section by discussing a parameterization of Kloosterman sums, a view that we will adopt later on. Let $R$ be a finite ring, and assume $\psi$ is a primitive character of $R$. By Lemma~\ref{lem: scale}, all the characters of $R$ are scalings of $\psi$, so every Kloosterman sum over $R$ has the form $K(a.\psi, b.\psi)$ for some $a,b\in R$. If both $a$ and $b$ are non-units then, by the reduction property, we may view the Kloosterman sum $K(a.\psi, b.\psi)$ as coming from a quotient of $R$. So we focus on the proper case when at least one of $a$ or $b$ is a unit in $R$. Next, thanks to unit-shifting, we can absorb the two variables, $a$ and $b$, into one. In summary, the Kloosterman sums over $R$ are essentially captured by the collection $\{K(a): a\in R\}$, where
\begin{align*}
K(a):=K(\psi, a.\psi)=\sum_{u\in R^\times} \psi(u+au^{-1}).
\end{align*}

Consider now the core case when the ring $R$ is local. Then the vanishing property, Proposition~\ref{prop: van-red}(i), implies that $K(a)=0$ whenever $a$ is a non-zero non-unit. Furthermore, $K(0)=K(\psi, \ct)$ is computed in Proposition~\ref{prop: rama}. So, in this case, the Kloosterman sums over $R$ are essentially captured by the restricted collection $\{K(a): a\in R^\times\}$.

The notation $K(a)$, introduced above, is convenient, but it suppresses the dependence on the chosen primitive character $\psi$. If $\rho$ is another primitive character, then $\rho=c.\psi$ for some unit $c\in R^\times$. We then have 
\[K(\rho, a.\rho)=K(c.\psi, (ac).\psi)=K(\psi, (ac^2).\psi)\] 
for any $a\in R$, by unit-shifting. Therefore both Kloosterman collections, the full $\{K(a): a\in R\}$, as well as the restricted $\{K(a): a\in R^\times\}$, are independent of the choice of primitive character $\psi$.

\section{Multiplicative characters}\label{sec: multichar}
A \emph{multiplicative character} of a finite ring $R$ is a character of the multiplicative group of units $R^\times$.

\subsection{Primitive multiplicative characters} Let $R$ be a finite local ring, with maximal ideal $M$. In order to define the multiplicative analogue of a conductor, recall first that each ideal $I\subseteq M$ defines a subgroup $1+I\subseteq R^\times$. We claim that, for each subgroup $H\subseteq R^\times$, there is a largest ideal $I\subseteq M$ so that $1+I\subseteq H$. For let $I_1, I_2\subseteq M$ be ideals satisfying $1+I_1, 1+I_2\subseteq H$. We have $1+I_1+I_2\subseteq (1+I_1)(1+I_2)$. Indeed, let $a_1\in I_1$ and $a_2\in I_2$; then 
\[1+a_1+a_2=(1+a_1)(1+(1+a_1)^{-1}a_2)\in (1+I_1)(1+I_2).\] 
It follows that the ideal $I_1+I_2\subseteq M$ also satisfies $1+I_1+I_2\subseteq H$, and the claim follows.

Let $\chi$ be a non-trivial multiplicative character of $R$. Applying the above remark to $H=\ker\chi$, we find that there is a largest ideal $I\subseteq M$ with the property that $\chi$ is trivial on the subgroup $1+I$; the ideal $I$ is the \emph{conductor} of $\chi$. By convention, the trivial multiplicative character $\ct$ has conductor $R$. 

\begin{defn}
A multiplicative character $\chi$ is said to be \emph{primitive} if its conductor is the zero ideal.
\end{defn}

Multiplicative characters of $R$ with conductor $I$ correspond to primitive multiplicative characters of the quotient ring $R/I$.

\subsection{The quadratic character} A prominent multiplicative character is the quadratic character, which is properly defined under some ring-theoretic requirements. 

Let $R$ be a local ring of odd characteristic; in other words, the residue field $R/M$ has odd characteristic. The \emph{quadratic character} of $R$, denoted by $\sigma$, is the multiplicative character induced by the quadratic multiplicative character of the residue field $R/M$. Thus $\sigma$ maps $R^\times$ onto $\pm 1$.

\begin{lem}
The quadratic character $\sigma$ enjoys the following properties:
\begin{itemize}
\item[(i)] The kernel of $\sigma$ is $(R^\times)^2$, the set of squares in $R^\times$. 
\item[(ii)] $\sigma$ has conductor $M$, and it is the only multiplicative character of $R$ of order $2$.
\item[(iii)] If $c\in R^\times$, then the equation $y^2=c$ has $1+\sigma(c)$ solutions.
\end{itemize}
\end{lem}

\begin{proof} (i) Squares in $R^\times$ map to squares in $(R/M)^\times$, so $(R^\times)^2\subseteq \ker \sigma$. Note that $\ker \sigma$ has index $2$ in $R^\times$. So if we show that $(R^\times)^2$ has index $2$ in $R^\times$, then the desired equality, $(R^\times)^2=\ker \sigma$, will follow. Consider then the squaring homomorphism $R^\times\to R^\times$, given by $u\mapsto u^2$. Its image is $(R^\times)^2$, and we now argue that its kernel is the two-element set $\{\pm 1\}$. If $u\in R$ satisfies $u^2=1$, then in the residue field $R/M$ we have $[u]^2=1$ and so $[u]=\pm 1$. Consider the case $[u]=1$. Then $u=1+m$ for some $m\in M$, and the equation $u^2=1$ turns into $m(2+m)=0$. As $2+m$ is a unit in $R$, we must have $m=0$, that is, $u=1$. The case $[u]=-1$ is similar.

(ii) Clearly, $\sigma$ is non-trivial, but it is trivial on $1+M$. So $M$ is the conductor of $\sigma$. It is also clear that $\sigma^2=\ct$. In order to prove uniqueness, let $\chi$ be a non-trivial multiplicative character of $R$ satisfying $\chi^2=\ct$. In particular, $\chi^2=\ct$ on the subgroup $1+M$. But $1+M$ has odd size; this follows by observing that each additive group $M^k/M^{k+1}$, for $k=1,2,\dots$, is naturally a vector space over the residue field $R/M$, and so $|M|$ is a power of $|R/M|$. Therefore, $\chi$ is trivial on $1+M$. This means that $\chi$ is induced from a multiplicative character of $R/M$ of order $2$. The only such character on $R/M$ is the quadratic character, and we conclude that $\chi=\sigma$. 

(iii) If $c\notin (R^\times)^2$, then $y^2=c$ has $0=1+\sigma(c)$ solutions. Assume now $c\in (R^\times)^2$; we show that $y^2=c$ has $2=1+\sigma(c)$ solutions. Put $c=d^2$, where $d\in R^\times$. Then $y^2=c$ is equivalent to $(yd^{-1})^2=1$, and we know from part (i) that the latter equation has precisely two solutions. \end{proof}

\begin{lem}\label{lem: solutions}
Consider the system
\[u+v=b, \quad uv=c\]
where $b\in R$ and $c\in R^\times$. If the discriminant $\Delta=b^2-4c$ is a unit in $R$, then the system has $1+\sigma(\Delta)$ unit solutions. In particular, for $b\in M$ there are $1+\sigma(-c)$ unit solutions.
\end{lem}

\begin{proof}
Note first that the condition $uv=c\in R^\times$ implies that the solutions are automatically units. The given system amounts to solving the equation $y(b-y)=c$, which is equivalent--thanks to $2\in R$ being a unit--to $(2y-b)^2=b^2-4c=\Delta$. This equation has $1+\sigma(\Delta)$ solutions. If, furthermore, $b\in M$, then $\Delta\equiv -4c$ mod $M$. Therefore $\sigma(\Delta)=\sigma(-4c)=\sigma(-c)$.
\end{proof}

If $R$ is a finite field, then the number of solutions of any quadratic equation can be given in terms of the quadratic character. We do not know to what extent this works over a general local ring $R$. As we have seen, one can count the number of solutions of some quadratic equations, but there is no apparent comprehensive formula involving the quadratic character that covers all quadratic equations.

\section{Gauss and Jacobi sums over finite rings}\label{sec: GJ}

\begin{defn}
Let $R$ be a finite ring. A \emph{Gauss sum} over $R$ is a character sum of the form
\[G(\psi, \chi)=\sum_{u\in R^\times} \psi(u)\: \chi(u)\]
where $\psi$ is an additive character of $R$, and $\chi$ is a multiplicative character of $R$. \end{defn}

When $\psi$ is implicit, we often write $G(\chi)$ instead of $G(\psi, \chi)$.

\begin{defn}
Let $R$ be a finite ring. A \emph{Jacobi sum} over $R$ is a character sum of the form
\[J(\chi, \eta)=\sum_{\substack{u,v\in R^\times \\ u+v=1}} \chi(u)\: \eta(v)\]
where $\chi$ and $\eta$ are multiplicative characters of $R$.
\end{defn}

Gauss and Jacobi sums satisfy the three basic properties that Kloosterman sums satisfy: factorization, and vanishing and reduction over local rings. Thanks to the factorization property, we may assume the following for the remainder of this section.

\begin{datum} $R$ is a local ring. 
\end{datum}

The next two lemmas collect the properties of Gauss sums and of Jacobi sums that will be needed in what follows. Note that the first part of item (i) in each lemma is an instance of the vanishing property.

\begin{lem}[\cite{L}]\label{lem: gauss}
Let $\psi$ be a primitive additive character, and $\chi$ a multiplicative character. Then the following hold.
\begin{itemize}
\item[(i)] If $\chi$ is neither trivial nor primitive, then $G(\chi)=0$. Moreover, if $\chi$ is trivial and $R$ is not a field, then $G(\chi)=0$ as well.
\item[(ii)] If $\chi$ is primitive, then $G(\chi)\: G(\overline{\chi})=\chi(-1)\: |R|$. Equivalently, $|G(\chi)|=|R|^{1/2}$.
\end{itemize}
\end{lem}

\begin{lem}[\cite{W}]\label{lem: jacobi}
Let $\chi$ and $\eta$ be multiplicative characters.
\begin{itemize}
\item[(i)] Assume $\eta$ is primitive. If $\chi$ is neither trivial nor primitive, then $J(\chi,\eta)=0$. Moreover, if $\chi$ is trivial and $R$ is not a field, then $J(\chi,\eta)=0$ as well.
\item[(ii)] Assume $\chi$ and $\eta$ are primitive. If $R$ is not a field, and $\chi\eta$ is neither trivial nor primitive, then $J(\chi,\eta)=0$.
\item[(iii)] If $\chi\eta$ is primitive, then $G(\chi)\: G(\eta)= J(\chi,\eta)\: G(\chi\eta)$. 
\end{itemize}
\end{lem}

In the last relation, we use the shorthand introduced in Lemma~\ref{lem: gauss}--namely, the Gauss sums are all taken with respect to a fixed primitive character.

We include a proof of Lemma~\ref{lem: jacobi}. Our approach is distinct from that of \cite{W}. (Also, we do not know what to make of the basic fact $(*)$ in \cite[p.286]{W}, a fact which gets used in most arguments of \cite{W}; perhaps there are some missing hypotheses, for it is false as stated.) The basic idea is to account for the vanishing of Jacobi sums by means of `orbital sumlets'.

\begin{lem}\label{lem: sumlet} Let $a\in R$, and consider the generalized Jacobi sum
\[J_a(\chi,\eta)=\sum_{\substack{u,v\in R^\times \\ u+v=a}} \chi(u)\: \eta(v).\] 
If there is a proper ideal $I$ so that $\sum_{r\in I} \chi(u-r)\: \eta(v+r)=0$ whenever $u,v\in R^\times$ satisfy $u+v=a$, then $J_a(\chi,\eta)=0$.
\end{lem}

The lemma is actually obvious, as soon as we adopt the following viewpoint: the indexing set $\{(u,v)\in R^\times\times R^\times: u+v=a\}$ admits a free action by the additive group $I$, given by $r.(u,v)=(u-r,v+r)$. So the total sum vanishes provided that the partial sum along each orbit vanishes.

\begin{proof}[Proof of Lemma~\ref{lem: jacobi}] 
(i) We have already pointed out that the first part is an instance of the vanishing property. Here we argue both parts via Lemma~\ref{lem: sumlet}, in a uniform way. Whether $\chi$ is neither trivial nor primitive, or $\chi$ is trivial and $R$ is not a field, there is  a proper ideal $I$ so that $\chi=\ct$ on $1+I$. For $u,v\in R^\times$ satisfying $u+v=1$, we have
\begin{align*}
\sum_{r\in I} \chi(u-r)\: \eta(v+r)=\chi(u) \sum_{r\in I} \eta(v+r)=\chi(u)\: \eta(v) \sum_{r\in I} \eta(1+r)=0
\end{align*}
as $\eta$ is primitive.

(ii) As $\chi\eta$ is not primitive, and $R$ is not a field, there is a proper ideal $I$ so that $\chi\eta=\ct$ on $1+I$, that is, $\chi(z)=\eta(z^{-1})$ for $z\in 1+I$. Again, we apply Lemma~\ref{lem: sumlet}. For $u,v\in R^\times$ satisfying $u+v=1$, we compute as follows:
\begin{align*}
\sum_{r\in I} \chi(u-r)\: \eta(v+r)&=\sum_{r\in I} \chi(u-ur)\: \eta(1-u+ur)\\
&=\chi(u) \sum_{z\in 1+I} \chi(z)\: \eta(1-uz)=\chi(u) \sum_{z\in 1+I}  \eta(z^{-1}-u)\\
&=\chi(u) \sum_{z\in 1+I}  \eta(z-u)=\chi(u) \sum_{r\in I}  \eta(v+r)\\
&=\chi(u)\:\eta(v) \sum_{r\in I}  \eta(1+r)=0.
\end{align*}
By way of explanation, let us simply point out that nearly every step is a change of variable. In the last step, we use the primitivity of $\eta$.

(iii) We expand
\begin{align*}
G(\chi)\: G(\eta)&= \sum_{u,v\in R^\times} \psi(u+v)\: \chi(u)\: \eta(v)\\
&=\sum_{a\in R} \psi(a) \sum_{\substack{u,v\in R^\times \\ u+v=a}} \chi(u)\: \eta(v)=\sum_{a\in R} \psi(a)\: J_a(\chi,\eta).\end{align*}
For $a\in R^\times$, a change of variable yields $J_a(\chi,\eta)=(\chi\eta)(a)\: J(\chi,\eta)$. On the other hand, we claim that $J_a(\chi,\eta)=0$ whenever $a\in M$. The desired formula will then follow:
\[G(\chi)\: G(\eta)=\sum_{a\in R^\times } \psi(a)\: (\chi\eta)(a)\: J(\chi,\eta)= J(\chi,\eta)\: G(\chi\eta).\]
The claim is clearly true when $a=0$:
\begin{align*}
J_0(\chi,\eta)=\sum_{u\in R^\times} \chi(u)\:\eta(-u)=\eta(-1) \sum_{u\in R^\times} (\chi\eta)(u)=0.
\end{align*}
Assume now that $a\neq 0$, and consider the proper ideal $I=(a)^\perp$. We apply Lemma~\ref{lem: sumlet}. Fix $u,v\in R^\times$ satisfying $u+v=a$; then $ur+vr=ar=0$ whenever $r\in I$. We compute
\begin{align*}
\sum_{r\in I} \chi(u-r)\: \eta(v+r)&=\sum_{r\in I} \chi(u-vr)\: \eta(v+vr)=\sum_{r\in I} \chi(u+ur)\: \eta(v+vr)\\
&=\chi(u)\: \eta(v) \sum_{r\in I} (\chi \eta)(1+r)=0,
\end{align*}
since $\chi\eta$ is primitive. This completes the proof of the claim.
\end{proof}

\section{Finite Frobenius rings}\label{sec: Frob} 
An existence issue laid dormant in Sections~\ref{sec: addchar} and ~\ref{sec: multichar}, namely, for which finite rings do primitive characters, additive or multiplicative, exist? The character sum considerations in Sections~\ref{sec: kloos} and ~\ref{sec: GJ} give much weight to this existence issue. We now discuss in some detail the relevant class of finite rings.

\subsection{Frobenius rings} Originally introduced by Nakayama \cite{Nak}, Frobenius rings are firmly established in the ring-theoretic canon \cite[Chapter 6]{Lam}. They form a class of good rings for a number of reasons, one being a very satisfactory duality theory for modules. In the present context, finite commutative rings, the following can be taken as definition: a ring $R$ is said to be \emph{Frobenius} if $R$ and $\widehat{R}$ are isomorphic as $R$-modules. It is, of course, generically true that $R$ and $\widehat{R}$ are isomorphic as abelian groups.

For our purposes, the character-theoretic perspective prevails. We will mainly use the following equivalent way of describing  Frobenius rings:

\begin{lem}\label{lem: frobprim}
A ring $R$ is Frobenius if and only if $R$ admits a primitive character.
\end{lem}

\begin{proof}
This is a further elaboration of the proof of Lemma~\ref{lem: scale} and the subsequent module-theoretic comments. Any $R$-module homomorphism $R\to\widehat{R}$ is of the form $h_\psi$ for some character $\psi$. The $R$-module homomorphism $h_\psi: R\to\widehat{R}$ is an isomorphism if and only if it is injective, which in turn is equivalent to $\psi$ being primitive.
\end{proof}

In view of the last part of Lemma~\ref{lem: condprod}, we deduce:

\begin{cor}
A product ring $R=R_1\times\dots \times R_n$ is Frobenius if and only if each factor $R_k$ is Frobenius.
\end{cor}

The next lemma provides two very useful properties of annihilators in Frobenius rings.

\begin{lem}\label{lem: Fann}
Let $R$ be a Frobenius ring, and let $I$ be an ideal of $R$. Then $|I||I^\perp|=|R|$, and $(I^\perp)^\perp=I$.
\end{lem}

\begin{proof} The statements are clearly true when $I=R$, so let us assume that $I\neq R$ in what follows. Let $\psi$ be a primitive character of $R$. The bijection $h_\psi:R\to \widehat{R}$, given by $a\mapsto a.\psi$, maps $I^\perp$ to $\mathrm{triv}(I)=\{\phi\in \widehat{R}: \phi=\ct \textrm{ on } I\}$, the subgroup of characters that are trivial on $I$.  Furthermore, if $h_\psi(a)=a.\psi\in \mathrm{triv}(I)$ for some $a\in R$, then $\psi$ is trivial on the ideal $aI$, hence $aI=(0)$ by the primitivity of $\psi$; thus $a\in I^\perp$. In summary, $h_\psi$ restricts to a bijection between $I^\perp\subseteq R$ and $\mathrm{triv}(I)\subseteq \widehat{R}$. On the other hand, it is a general fact that $\mathrm{triv}(I)$ is in correspondence with the additive dual of the quotient ring $R/I$. Therefore $|I^\perp|=|\mathrm{triv}(I)|=|R|/|I|$, as desired. 

For the second fact, we note first the generic inclusion $I\subseteq (I^\perp)^\perp$. Under the Frobenius hypothesis we also have $|I||I^\perp|=|R|=|I^\perp||(I^\perp)^\perp|$, whence $|I|=|(I^\perp)^\perp|$. We conclude that $I=(I^\perp)^\perp$.
\end{proof}

\subsection{Local Frobenius rings} Deciding whether a ring is Frobenius reduces to deciding whether the local factors of the ring are Frobenius. For local rings, the following ideal-theoretic criteria hold:

\begin{lem}\label{lem: Floc}
Let $R$ be a local ring, with maximal ideal $M$. Then the following are equivalent:
\begin{itemize}
\item[(i)] $R$ is Frobenius;
\item[(ii)] $R$ has a unique minimal ideal;
\item[(iii)] $M^\perp$ is a minimal ideal;
\item[(iv)] $M^\perp$ is a principal ideal.
 \end{itemize}
\end{lem}

The proof of the lemma relies, in part, on a property of $M^\perp$ that holds in any local ring $R$. The fact is that $M^\perp$ acts as a bottleneck for minimal ideals.

\begin{lem}
Let $R$ be a local ring with maximal ideal $M$. Then each minimal ideal of $R$ is contained in $M^\perp$. In particular, the ideal $M^\perp$ is non-zero.
\end{lem}

\begin{proof} 
A minimal ideal must be principal and, by definition, non-zero. So let $(r)$ be a minimal ideal of $R$. For each $a\in M$, we have $(ra)\subseteq (r)$, so either $(ra)=(0)$ or $(ra)=(r)$ by minimality. If the latter holds, then $r=rab$ for some $b\in R$. But then $r(1-ab)=0$ and $1-ab\in 1+M$ is a unit, so $r=0$, a contradiction. Therefore the former holds: $ra=0$. This means that $ r\in M^\perp$, as desired.
\end{proof}

\begin{proof}[Proof of Lemma~\ref{lem: Floc}] (i)$\: \Rightarrow\: $(iii) Let $I$ be any minimal ideal of $R$. Then $I\subseteq M^\perp$, equivalently, $M\subseteq I^\perp$. Moreover, $I^\perp \neq R$, as $I\neq (0)$, so we deduce that $I^\perp=M$. Now $R$ is Frobenius, by hypothesis, so it follows that $I=M^\perp$ thanks to the double annihilator property of Lemma~\ref{lem: Fann}. Therefore $M^\perp$ is minimal. 

(iii)$\:\Rightarrow\:$(ii) Any minimal ideal is contained in $M^\perp$. So if $M^\perp$ is minimal, then it is the unique minimal ideal of $R$. 

(ii)$\:\Rightarrow\:$(i) Assume $R$ has a unique minimal ideal $I$. The non-primitive characters of $R$ are precisely those which are trivial on $I$. So there are $|\widehat{R/I}|=|R/I|<|R|$ non-primitive characters of $R$. It follows that $R$ admits primitive characters. 

(iii)$\:\Rightarrow\:$(iv) Any minimal ideal is principal. 

(iv)$\:\Rightarrow\:$(iii) Let $M^\perp=(a)$. Then $M^\perp=\{ra: r\in R\}=\{0\}\cup \{ua: u\in R^\times\}$, as $a$ annihilates each element of $M$. If $I$ is non-zero ideal contained in $M^\perp$, then $I$ must contain a unit multiple of $a$ and so $I=M^\perp$. Therefore $M^\perp$ is a minimal ideal.
\end{proof}

The unique minimal ideal $M^\perp$ of a local Frobenius ring $R$ will often show up in what follows. Note that, typically, we have the inclusion $M^\perp\subseteq M$. In the degenerate case when $R$ is a field, the inclusion is reversed.

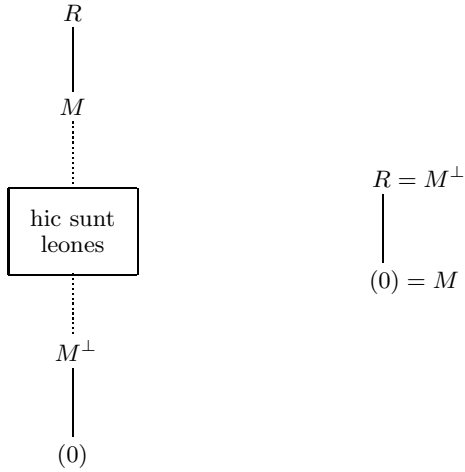
\begin{figure*}[!h]
\begin{align*}
\begin{minipage}[b]{0.48\linewidth}
\xymatrix{
&R &\\
&M \ar@{-}[u] &\\
  & *+++[F]{\txt{hic sunt \\%
leones}}\ar@{.}[u]&\\
& \; M^\perp \ar@{.}[u]&\\
&(0) \ar@{-}[u] &\\
}
\end{minipage}
\begin{minipage}[b]{0.48\linewidth}
\xymatrix{
&&\\
&&\\
&\qquad \:  \enskip\:  R =M^\perp &\\
&\qquad \enskip (0)=M \ar@{-}[u] &\\
}
\end{minipage}
\end{align*}\caption{The maximal ideal and its annihilator in a local Frobenius ring: a non-field (left) versus a field (right).}
\end{figure*}

For local rings, Lemma~\ref{lem: frobprim} has a multiplicative counterpart.

\begin{lem}\label{lem: Fmult}
Let $R$ be a local ring. Then $R$ is Frobenius if and only if $R$ admits a primitive multiplicative character.
\end{lem}

\begin{proof} Let us focus on the non-trivial case when $R$ is not a field.

Assume $R$ is Frobenius. As in Lemma~\ref{lem: Floc} above, we argue existence by counting. As $M^\perp$ is the unique minimal ideal, a multiplicative character $\chi$ is not primitive precisely when $\chi$ is trivial on $1+M^\perp$. So there are $|R^\times/(1+M^\perp)|=|R^\times|/|M^\perp|$ non-primitive multiplicative characters, leaving $|R^\times|-|R^\times|/|M^\perp|>0$ primitive ones.

Conversely, assume $R$ admits a primitive multiplicative character $\chi$. We argue that the ideal $M^\perp$ is principal; Lemma~\ref{lem: Floc} will then imply that $R$ is Frobenius. It suffices to show the following: if $a,b\in M^\perp$ are non-zero, then $b=ua$ for some $u\in R^\times$.

For $a\in M^\perp$, put $\phi_a(r)=\chi(1+ar)$ for $r\in R$. Using the fact that $a^2=0$, we see that $\phi_a$ is an additive character: 
\[\phi_a(r_1)\phi_a(r_2)=\chi((1+ar_1)(1+ar_2))=\chi(1+a(r_1+r_2))=\phi_a(r_1+r_2).\] 
Clearly, $\phi_a$ is trivial on $M$. It is easily seen that $\phi_a$ is trivial if and only if $a=0$, thanks to $\chi$ being primitive. Thus, for $a\neq 0$ we may view $\phi_a$ as the lift of a non-trivial additive character $\phi'_a$ of the residue field $R/M$. Now, on $R/M$, any two non-trivial additive characters are equivalent via scaling. So, if $a,b\in M^\perp$ are non-zero, then $\phi'_b=v.\phi'_a$ for some $v\in (R/M)^\times$. This lifts to $\phi_b=u.\phi_a$ for some $u\in R^\times$, that is, $\phi_{b-ua}=\ct$. It follows that $b=ua$, as desired.
\end{proof}

\subsection{Some examples} Finite principal rings are Frobenius. Indeed, they are products of local principal rings, and these are Frobenius by a direct application of Lemma~\ref{lem: Floc} above. Examples of finite principal rings include:
\begin{itemize}
\item[$\cdot$] finite fields;
\item[$\cdot$] finite quotients of principal domains; this subclass includes, notably, the modular rings $\Z/(n)$, and their polynomial analogues $\F[X]/(f)$ where $\F$ is a finite field;
\item[$\cdot$] more generally, finite quotients of Dedekind domains, such as quotients of algebraic rings of integers, and their function field analogues.
\end{itemize}

Finite principal rings may be deemed the most important examples of Frobenius rings, as they account for the number-theoretic considerations. On the other hand, finite principal rings are also some of the simplest examples, for their algebraic structure is rather tame  and easy to parameterize. So let us add a more exotic example of a Frobenius ring. It is based on a certain extension construction, which we recall first. Given a finite ring $R$, and an $R$-module $\Lambda$, one constructs a new finite ring $R\ltimes \Lambda$ as follows: the additive structure is that of the product $R\times \Lambda$, while the multiplication is given by $(r,\lambda)(r',\lambda')=(rr',r\lambda'+r'\lambda)$. 

\begin{ex} Let $R$ be a finite ring, and consider the character $R$-module $\widehat{R}$. Then the finite ring $R\ltimes \widehat{R}$ is Frobenius. 

Indeed, we produce a primitive character on $R\ltimes \widehat{R}$. Consider the map $\rho: R\ltimes \widehat{R}\to \C^*$, given by $\rho(r,\psi)=\psi(1)$. Then $\rho$ is a homomorphism, for it is the composition $R\times \widehat{R}\to \widehat{R}\to \C^*$, where the first map is the natural projection, and the second map is evaluation at $1\in R$. In order to see that $\rho$ is primitive, it suffices to argue that the only principal ideal of $R\ltimes \widehat{R}$, on which $\rho$ is trivial, is the zero ideal. So let $(r,\psi)\in R\ltimes \widehat{R}$, and assume that $\rho=\ct$ on the ideal generated by $(r,\psi)$. We wish to obtain $(r,\psi)=(0,\ct)$, the zero element in $R\ltimes \widehat{R}$. For all $r'\in R$ and all $\psi'\in \widehat{R}$, we have
\[1=\rho((r',\psi')(r,\psi))=\rho(rr', (r.\psi')(r'.\psi))=(r.\psi')(1)\cdot(r'.\psi)(1)=\psi'(r)\psi(r').\]
Firstly, take $\psi'=\ct$; then $\psi(r')=1$ for all $r'\in R$, that is $\psi=\ct$. Secondly, take $r'=0$; then $\psi'(r)=1$ for all $\psi'\in \widehat{R}$, therefore $r=0$.
\end{ex}

\begin{notes}
This section is a synopsis of known facts. They are drawn from several sources and, for the sake of cohesion, we have decided to include proofs. Further work on character sums over finite Frobenius rings is likely to find this synopsis useful. 

Much of the above discussion is based on \cite[Ch.6]{Lam} and \cite{Hon}. The `exotic' example is, essentially, \cite[Ex.16.60]{Lam}. See \cite[\S 16.G]{Lam} for an answer to the sensible question `Why Frobenius?'. Lemmas~\ref{lem: Fann} and ~\ref{lem: Floc} go back to \cite{L}, even \cite{Nak} to some extent; compare also \cite[Lems.3.2, 3.3]{Sz}. In \cite[Thm.1]{Hon}, it is furthermore shown that each one of the annihilator properties in Lemma~\ref{lem: Fann} characterizes, in fact, Frobenius rings. The double annihilator property, $(I^\perp)^\perp=I$, in a Frobenius ring is analogous to the double orthogonal property, $W^{\perp\perp}=W$, in a Hilbert space. As such, it expresses an algebraic completeness of Frobenius rings. Lemma~\ref{lem: Fmult} is a result from \cite{W}, though the forward direction already appeared in \cite{L}. 
\end{notes}

\section{The Selberg-Kuznetsov identity}\label{sec: SK}
The Selberg-Kuznetsov identity for classical Kloosterman sums says that
\begin{align}\label{eq: SK}
S(m,n; q)=\sum_{d|m,n,q} d\: S(1,mn/d^2; q/d)
\end{align}
for any triple of positive integers $m$, $n$, $q$. We prove the following generalization.

\begin{thm}\label{thm: gSK}
Let $R$ be a finite Frobenius ring, and let $\psi$ be a primitive character of $R$. Then for all $a,b\in R$ we have
\begin{align}\label{eq: gSK}
K(a.\psi,b.\psi; R)=\sum_{(d)\supseteq (a), (b)} |d^\perp|\: K(d.\psi,(ab/d^2).(d.\psi); R/d^\perp).
\end{align}
\end{thm}

Let us explain the meaning of the formula \eqref{eq: gSK}. The left-hand side is a Kloosterman sum over $R$, while the right-hand side involves  Kloosterman sums over certain quotients of $R$. The summation is indexed by the principal ideals of $R$ which contain the principal ideals generated by $a$ and $b$. Each summand involves a chosen generator per principal ideal. The scaled character $d.\psi$, it has conductor $d^\perp$ and so we can view it as a primitive character of the quotient $R/d^\perp$. If $(d)\supseteq (a), (b)$, then $a=da_0$ and $b=db_0$ for some $a_0,b_0\in R$; the quotient $ab/d^2$ is to be interpreted as $a_0b_0\in R$. Note that $a_0$ and $b_0$ are well-defined mod $d^\perp$, so $a_0b_0$ is well-defined mod $d^\perp$ as well.

Each summand is, in fact, independent of the choice of a generator. Indeed, $d^\perp=(d)^\perp$ only depends on the ideal $(d)$. If $d'$ is another generator of $(d)$, then $d'$ and $d$ are unit multiples of each other, say $d'=ud$ for some $u\in R^\times$. (This fact, while not true for general commutative rings, does hold and it is not hard to prove in our finite context.) The Kloosterman sums in terms of $d'$ and $d$ are then equal, by shifting the unit $u$.

Let us also explain how the formula \eqref{eq: gSK} turns into the Selberg-Kuznetsov identity \eqref{eq: SK} in the case $R=\Z/q\Z$. We let $a,b\in \Z$, and we use the standard primitive character $e_q$ in place of $\psi$. Up to a slight abuse of notation, the left-hand side of \eqref{eq: gSK} is the classical Kloosterman sum $S(a,b; q)$. Turning to the right-hand side, we first note that the  principal ideals of $\Z/q\Z$ are parameterized by the positive divisors of $q$. If $d$ is a positive divisor of $q$, and $a\in \Z$, then $(d)\supseteq (a)$ in $\Z/q\Z$ if and only if $d|a$ in $\Z$. Thus, the summation is indexed by the positive divisors $d$ of $q$ which also divide $a$ and $b$. Further, if $d$ is a positive divisor of $q$, then $d^\perp=(q/d)$ in $\Z/q\Z$. Thus $|d^\perp|=d$, and $R/d^\perp=(\Z/q\Z)/(q/d)$ is naturally identified with $\Z/(q/d)\Z$. Finally, the scaled character $d.e_q$ on $\Z/q\Z$ corresponds to the standard primitive character $e_{q/d}$ on $\Z/(q/d)\Z$. Thus, the Kloosterman sum associated to each $d$ is indeed $S(1,mn/d^2; q/d)$. 

We point out that the value $d=0$ is allowed in \eqref{eq: gSK}. This is an extremal case that only arises when $a=b=0$. Then $d^\perp=R$, and the quotient $R/d^\perp$ is the degenerate zero ring. Reasonably, the Kloosterman sum over the zero ring equals $1$. A similar comment is needed for the classical Selberg-Kuznetsov identity \eqref{eq: SK}; there, the degeneration occurs when $d=q$.

\begin{proof}[Proof of Theorem~\ref{thm: gSK}]
The first observation is that we may reduce to the case when $R$ is local. The general case of a finite ring $R$ will follow by the multiplicative behaviour of the Kloosterman sums, as well as of all the other algebraic ingredients involved in \eqref{eq: gSK}. See Section~\ref{sec: mom} for an instance of this argument.

Assume, then, that $R$ is local, with maximal ideal $M$. Recall that $M^\perp$ is the minimal ideal, and it is principal. Put
\[K_{(d)}:=K(d.\psi, (ab/d^2).(d.\psi); R/d^\perp)\]
and recall that $d.\psi$ is a primitive character on $R/d^\perp$ for $d\neq 0$.

The first case we consider is when $a=0$ or $b=0$. Then $K_{(d)}=K(d.\psi, \ct; R/d^\perp)$. For $d\neq 0$, we have that $K_{(d)}$ vanishes unless $R/d^\perp$ is a field. The latter amounts to $d^\perp=M$, in other words, $(d)=M^\perp$. In this case, we have $K_{(d)}=-1$, and so $|d^\perp|\: K_{(d)}=-|M|$. We now turn to verifying \eqref{eq: gSK}, and we do so by computing the right-hand side. If $a=b=0$ then 
\begin{align*}
\sum_{(d)} |d^\perp|\: K_{(d)}=|R|+ \sum_{(d)\neq (0)} |d^\perp|\: K_{(d)}=|R|-|M|=|R^\times|=K(\ct,\ct; R).
\end{align*}  
If, say, $b=0$ but $a\neq 0$, then the right-hand side of \eqref{eq: gSK} is
\begin{align*}
\sum_{(d)\supseteq (a)} |d^\perp|\: K_{(d)}&=\sum_{M^\perp=(d)\supseteq (a)} -|M|=-|M| \: [(a)= M^\perp]=K(a.\psi,\ct; R).
\end{align*}  

The second case is when $a,b\neq 0$. Consider again the Kloosterman sum
\[K_{(d)}=K(d.\psi, (a_0b_0).(d.\psi); R/d^\perp)\]
corresponding to $(d)\supseteq (a), (b)$, where $a=da_0$ and $b=db_0$ as explained above. We determine which principal ideals $(d)$ have $K_{(d)}\neq 0$. Firstly, the character $(a_0b_0).(d.\psi)$ must be either primitive or trivial on $R/d^\perp$. Now $(a_0b_0).(d.\psi)$ is primitive if and only if $a_0b_0$ is a unit mod $d^\perp$. This means that both $a_0$ and $b_0$ are units mod $d^\perp$, which is equivalent to $a$ and $b$ being unit multiples of $d$. The ideal-theoretic reformulation is $(d)=(a)=(b)$. Note, furthermore, that for $(d)=(a)=(b)$ we have
\[K_{(d)}=K_{(a)}=K(a.\psi, b.\psi; R/a^\perp).\]
On the other hand, if the character $(a_0b_0).(d.\psi)$ is trivial, then $a_0b_0=0$ mod $d^\perp$. Then at least one of $a_0$ or $b_0$ is a non-unit mod $d^\perp$, and this translates into the fact that at least one of the inclusions $(d)\supseteq (a), (b)$ is proper. Recall, however, from the very first case above, that $K_{(d)}=K(d.\psi, \ct; R/d^\perp)$ is non-vanishing only for $(d)=M^\perp$. As $M^\perp$ is minimal, and $a,b\neq 0$, the inclusions $(d)\supseteq (a), (b)$ cannot be proper. In summary, the trivial character situation cannot be realized. 

In this second case, then, the right-hand side of \eqref{eq: gSK} is
\begin{align*}
\sum_{(d)\supseteq (a), (b)} |d^\perp|\: K_{(d)}= |a^\perp|\: K(a.\psi, b.\psi; R/a^\perp) \: [(a)= (b)]=K(a.\psi,b.\psi; R).
\end{align*}
Indeed, note that $(a)$ and $(b)$ are the conductors of the scaled characters $a.\psi$ and $b.\psi$ on $R$. So the latter equality captures the fundamental properties of Kloosterman sums: reduction, when $(a)= (b)$, respectively vanishing when $(a)\neq (b)$.
\end{proof}

The previous proof has a drawback: it merely verifies the identity \eqref{eq: gSK}. We are hopeful that a more conceptual proof exists, one which better explains why \eqref{eq: gSK} should be true.

\section{Twisted Kloosterman sums}\label{sec: twistedK}
In this section, we introduce a generalization of Kloosterman sums. The additional ingredient is a fixed multiplicative character, which `twists' the usual Kloosterman sum. By now, we know that the appropriate algebraic context is as follows.
 
 \begin{datum}
$R$ is a Frobenius ring, and $\psi$ is a fixed primitive character of $R$.
\end{datum}
 
\begin{defn}
A \emph{twisted (or generalized) Kloosterman sum} over $R$ is a character sum of the form
\[K_\tau(a)=\sum_{u\in R^\times} \tau(u)\:\psi(u+au^{-1})\]
where $a\in R$, and $\tau$ is a multiplicative character of $R$. 
\end{defn}

This definition takes the parameterized perspective on the usual Kloosterman sums, cf. end of Section~\ref{sec: kloos}, as its starting point. We usually think of the multiplicative character $\tau$ as being fixed, whereas $a\in R$ varies. We then refer to $\tau$ as the \emph{twist}.

Let us start with some basic observations concerning twisted Kloosterman sums. Firstly, we have
\[K_\tau(0)=G(\tau),\]
the Gauss sum of $\tau$ with respect to $\psi$. Secondly, twisted Kloosterman sums are no longer real numbers, as the usual Kloosterman sums were. Two useful formulas for the conjugate are
\begin{align*}
\overline{K_\tau(a)}&=\tau(-1)\:K_{\overline{\tau}}(a) \quad \textrm{for } a\in R,\\
\overline{K_\tau(a)}&=\tau(-a^{-1})\:K_{\tau}(a) \quad \textrm{for } a\in R^\times.
\end{align*}
Thirdly, the following hold.

\begin{prop}\label{prop: tK1}
For any twist $\tau$, we have the zero-mean relation
\begin{align*}
\sum_{a\in R} K_\tau(a)=0.
\end{align*}
For any two twists $\tau$ and $\eta$, we have the orthogonality relation
\begin{align*}
\sum_{a\in R} K_\tau(a)\: \overline{K_\eta(a)}=|R||R^\times| \: [\tau=\eta].
\end{align*}
\end{prop}

\begin{proof} Indeed:
\begin{align*}
\sum_{a\in R} K_\tau(a)&=\sum_{a\in R} \sum_{u\in R^\times} \tau(u)\: \psi(u)\:\psi(au^{-1})=\sum_{u\in R^\times} \tau(u)\:\psi(u) \sum_{a\in R}  \psi(au^{-1})=0
\end{align*}
since each inner sum clearly vanishes.

Turning to the second relation, for each $a\in R$ we write
\begin{align*}
K_\tau(a)\: \overline{K_\eta(a)}=\sum_{u,v\in R^\times} \tau(u)\:\overline{\eta}(v)\:\psi(u-v) \:\psi(a(u^{-1}-v^{-1}).
\end{align*}
Then
\begin{align*}
\sum_{a\in R} K_\tau(a)\: \overline{K_\eta(a)}&=\sum_{u,v\in R^\times} \tau(u)\:\overline{\eta}(v)\:\psi(u-v) \:\sum_{a\in R} \psi(a(u^{-1}-v^{-1})\\
&=\sum_{u,v\in R^\times} \tau(u)\:\overline{\eta}(v)\:\psi(u-v) \:|R|\: [u=v]\\
&=|R|\sum_{u\in R^\times} (\tau\overline{\eta})(u)=|R|\:|R^\times|\: [\tau=\eta]
\end{align*}
as claimed.
\end{proof}

Fourthly, we look into how the twisted Kloosterman sums depend on the fixed primitive character $\psi$. Let $\rho=c.\psi$, where $c\in R^\times$, be another primitive character of $R$. Keeping track of the primitive character, we write:
\begin{align*}
K_\tau(a; \rho)&=\sum_{u\in R^\times} \tau(u)\:\rho (u+au^{-1})=\sum_{u\in R^\times} \tau(u)\:\psi(cu+cau^{-1}) \\
&=\sum_{u\in R^\times} \tau(c^{-1}u)\:\psi(u+c^2au^{-1})=\overline{\tau}(c)\: K_\tau(ca^2; \psi).
\end{align*}
Therefore the collection of twisted Kloosterman sums $\{K_\tau(a): a\in R\}$ is rotated upon change of the primitive character $\psi$. In particular, the collection of absolute values $\{|K_\tau(a)|: a\in R\}$ is independent of $\psi$, and so intrinsic to the ring $R$. The same conclusions apply to the restricted collections $\{K_\tau(a): a\in R^\times\}$ and $\{|K_\tau(a)| : a\in R^\times\}$.

From now on, our study of twisted Kloosterman sums will mostly focus on non-primitive twists. This case encompasses the usual Kloosterman sums, but also the Sali\'e sums over local Frobenius rings which are not fields. In fact, the non-field assumption will appear often since, over a field, the only non-primitive twist is the trivial one.

We start with an important vanishing property. The first item is in line with the vanishing property of the usual Kloosterman sums, but the second item is a bit more surprising. It extends a well-known property of usual Kloosterman sums and of Sali\'e sums in the classical case $R=\Z/q\Z$, $q$ a proper power of an odd prime.

\begin{thm}\label{thm: tK2}
Assume $R$ is local, but not a field, and let $\tau$ be a non-primitive twist. Then the following hold:
\begin{itemize}
 \item[(i)] $K_\tau(a)=0$ whenever $a\notin R^\times$.
 \item[(ii)] If $R$ has odd characteristic then, furthermore, $K_\tau(a)=0$ whenever $a\notin (R^\times)^2$. 
 \end{itemize}
\end{thm}

\begin{proof} Let $a\in R$. In order to show that $K_\tau(a)=0$, in suffices to show that each partial sum of the form
\begin{align*}
\sum_{u\in u_0(1+M^\perp)} \tau(u)\:\psi(u+au^{-1})
\end{align*}
where $u_0\in R^\times$, vanishes. The above sumlet corresponds to a coset of the subgroup $1+M^\perp\subseteq R^\times$. 

By hypothesis, $\tau$ is trivial on the subgroup $1+M^\perp$. Therefore $\tau(u)=\tau(u_0)$ whenever $u\in u_0(1+M^\perp)$. So we have:
\begin{align*}
\sum_{u\in u_0(1+M^\perp)} \tau(u)\:\psi(u+au^{-1})&=\tau(u_0)\sum_{r\in M^\perp} \psi\big(u_0(1+r)+a(u_0(1+r))^{-1}\big)\\
&=\tau(u_0) \sum_{r\in M^\perp} \psi\big(u_0(1+r)+au_0^{-1}(1-r)\big)\\
&=\tau(u_0)\:\psi(u_0+au_0^{-1})\sum_{r\in M^\perp} \psi(r(u_0-au_0^{-1})).
\end{align*}
On the way, we have crucially used the fact that $r^2=0$, and so $(1+r)^{-1}=1-r$, whenever $r\in M^\perp$. Consider now the latter sum of the above computation. By Lemma~\ref{lem: primsum}, it vanishes unless $u_0-au_0^{-1}\in (M^\perp)^\perp=M$, equivalently, $a\in u_0^{2}(1+M)$. But this cannot happen: in case (i), because $u_0^{2}(1+M)\subseteq R^\times$, while $a\notin R^\times$; in case (ii), because $u_0^{2}(1+M)\subseteq (R^\times)^2$, while $a\notin (R^\times)^2$.
\end{proof}

\section{Moments}\label{sec: mom}
 \begin{datum}
$R$ is a Frobenius ring, and $\psi$ is a fixed primitive character of $R$.
\end{datum} 

A way into understanding the collections of Kloosterman sums $\{K(a): a\in R\}$ and $\{K(a) : a\in R^\times\}$, is to consider the \emph{moments} 
\begin{align*}
\sum_{a\in R} K(a)^k, \qquad \sum_{a\in R^\times} K(a)^k
\end{align*}
for $k=1, 2, \dots$. In what follows, we refer to the first as a \emph{full moment}, respectively \emph{unitary moment} to the restricted one. Note that the moment values are intrinsic to the ring $R$, in the sense that they do not depend on the choice of primitive character $\psi$, cf. discussion ending Section~\ref{sec: kloos}. 

The full moments and the unitary moments are multiplicative as functions of the ring $R$. To see this, let us put
\[\mu^\times_k(R)=\sum_{a\in R^\times} K(a)^k\]
and let us show that $\mu_k^\times(R_1\times R_2)=\mu_k^\times(R_1)\:\mu_k^\times(R_2)$. If $\psi_1$ and $\psi_2$ are primitive characters of $R_1$, respectively $R_2$, then $\psi_1\times \psi_2$ is a primitive character of $R_1\times R_2$, and we have
\begin{align*}
\mu_k^\times(R_1)\:\mu_k^\times(R_2)&=\Big(\sum_{a_1\in R_1^\times} K(\psi_1, a_1.\psi_1)^k\Big)\Big(\sum_{a_2\in R_2^\times} K(\psi_2, a_2.\psi_2)^k\Big)\\
&= \sum_{(a_1,a_2)\in R_1^\times\times R_2^\times} \big(K(\psi_1, a_1.\psi_1)\: K(\psi_2, a_2.\psi_2)\big)^k\\
&=\sum_{(a_1,a_2)\in (R_1\times R_2)^\times} K\big(\psi_1\times\psi_2, (a_1,a_2).(\psi_1\times\psi_2)\big)^k\\
&=\mu_k^\times(R_1\times R_2).
\end{align*}
Along the way, we have used the multiplicative property of Kloosterman sums. A similar argument applies to the full moments.

Turning now to general context of twisted Kloosterman sums, we consider the moments
 \begin{align*}
\sum_{a\in R} K_\tau(a)^k, \qquad \sum_{a\in R^\times} K_\tau(a)^k
\end{align*}
as before, but also the \emph{absolute moments}
 \begin{align*}
\sum_{a\in R} \big|K_\tau(a)\big|^k, \qquad \sum_{a\in R^\times} \big|K_\tau(a)\big|^k.
\end{align*}

For instance, by Proposition~\ref{prop: tK1} we have
\begin{align*}
\sum_{a\in R} K_\tau(a)=0, \qquad \sum_{a\in R} \big|K_\tau(a)\big|^2=|R||R^\times|.
\end{align*}
 
The absolute moments of twisted Kloosterman sums are independent of the choice of $\psi$, while the moments themselves have a minor phase dependence. Moments, absolute or not, factor according to the factorization of the underlying ring $R$. This can be shown in a manner similar to the one exhibited above, for the moments of the usual Kloosterman sums, the adjustment being that the twist $\tau$ has to be factored as well. 

And so we focus, again, on the local case. In this case, and for a non-primitive twist, the full and the unitary moments are nearly the same.

\begin{lem}\label{lem: link}
Assume $R$ is local, and let $\tau$ be a non-primitive twist. Then
\begin{align*}
\sum_{a\in R} K_\tau(a)^k&=K_\tau(0)^k+\sum_{a\in R^\times} K_\tau(a)^k,\\
\sum_{a\in R} \big|K_\tau(a)\big|^k&=\big|K_\tau(0)\big|^k+\sum_{a\in R^\times} \big|K_\tau(a)\big|^k.
\end{align*}
Furthermore, if $R$ is not a field then $K_\tau(0)=0$.
\end{lem}

\begin{proof}
If $R$ is a field, the relations are trivially true. If $R$ is not a field, then the relations follow from the vanishing property of Theorem~\ref{thm: tK2}(i). Moreover, in this case $K_\tau(0)=G(\tau)$ vanishes as well, by Lemma~\ref{lem: gauss}(i).
\end{proof}

Our main interest, fuelled by classical results, is in the unitary moments, absolute or not. The usefulness of the above lemma stems from the observation that full moments are sometimes easier to manipulate and compute. This full-unitary strategy is exploited in the following computation. 

\begin{thm}\label{thm: 3mom} Assume $R$ is local, but not a field. Assume also that $R$ has odd characteristic, and $3$ is a unit in $R$. Then, for any non-primitive twist $\tau$, we have
\begin{align*}
\sum_{a\in R^\times} K_\tau(a)^3=0.
\end{align*}
\end{thm}

\begin{proof}
For each $a\in R$, we have
\begin{align*}
K_\tau(a)^3=\sum_{u,v,w\in R^\times} \tau(uvw)\: \psi(u+v+w)\:\psi(a(u^{-1}+v^{-1}+w^{-1}))
\end{align*}
and so
\begin{align*}
\sum_{a\in R} K_\tau(a)^3&= \sum_{u,v,w\in R^\times} \tau(uvw)\:\psi(u+v+w)\:\sum_{a\in R}\psi(a(u^{-1}+v^{-1}+w^{-1}))\\
&=|R| \sum_{\substack{u,v,w\in R^\times \\ u^{-1}+v^{-1}+w^{-1}=0}} \tau(uvw)\:\psi(u+v+w).
\end{align*}
The left-hand side falls under the scope of Lemma~\ref{lem: link}. On the right, the indexing set $\{(u,v,w): u,v,w\in R^\times, u^{-1}+v^{-1}+w^{-1}=0\}$ can be parameterized as $\{(bu,bv,-buv): u,v,b\in R^\times, u+v=1\}$. We continue our computation as follows:
\begin{align*}
\frac{1}{|R|}\sum_{a\in R^\times} K_\tau(a)^3&=\sum_{b\in R^\times}\sum_{\substack{u,v\in R^\times\\ u+v=1}} \tau(-b^3(uv)^2)\:\psi (b-buv)\\
&=\sum_{c\in R^\times}\sum_{\substack{u,v\in R^\times\\ u+v=1,\: uv=c}} \sum_{b\in R^\times}\tau(-b^3c^2)\:\psi (b-bc).
\end{align*}
The inner-most sum is, essentially, a Gauss sum:
\begin{align*}
\sum_{b\in R^\times}\tau(-b^3c^2)\:\psi (b-bc)&=\tau(-c^2)\sum_{b\in R^\times}\tau^3(b)\:\psi ((1-c)b)\\
&=\tau(-c^2)\: G(\tau^3, (1-c).\psi).
\end{align*}
As $\tau$ is a non-primitive multiplicative character, so is $\tau^3$. By Lemma~\ref{lem: gauss}(i), we know that $G(\tau^3, (1-c).\psi)=0$ whenever $(1-c).\psi$ is a primitive character, that is, whenever $1-c\in R^\times$. So the above inner-most sum vanishes whenever $1-c\in R^\times$. We can then restrict to the case when $1-c\in M$, that is, $c\in 1+M\subseteq R^\times$. 

The next counting step comes courtesy of Lemma~\ref{lem: solutions}. Consider a system $u+v=1$, $uv=c$ with $c\in 1+M$. Then the discriminant $\Delta$ is a unit in $R$, for $\Delta=1-4c\equiv -3$ mod $M$, and $3\in R^\times$ by assumption. Hence $\sigma(\Delta)=\sigma(-3)$, and the system has $1+\sigma(-3)$ unit solutions. Returning to our computation, we get
\begin{align*}
\frac{1}{|R|}\sum_{a\in R^\times} K_\tau(a)^3&=\sum_{c\in 1+M}\sum_{\substack{u,v\in R^\times\\ u+v=1,\: uv=c}} \sum_{b\in R^\times}\tau(-b^3c^2)\:\psi (b-bc)\\
&= (1+\sigma(-3))\: \sum_{c\in 1+M}\sum_{b\in R^\times}\tau(-b^3c^2)\:\psi (b-bc)\\
&= (1+\sigma(-3))\: \sum_{m\in M}\sum_{b\in R^\times}\tau(-b^3(1+m)^2)\:\psi (-bm)\\
&= (1+\sigma(-3))\: \sum_{b\in R^\times}\tau(-b^3)\sum_{m\in M}\tau^2(1+m)\:(-b.\psi) (m).
\end{align*}
Now $\tau^2$ is a non-primitive multiplicative character, while $-b.\psi$ is a primitive character for each $b\in R^\times$. The desired vanishing, $\sum_{a\in R^\times} K_\tau(a)^3=0$, will follow as soon as we establish the following claim: for any non-primitive multiplicative character $\chi$, and any primitive character $\phi$, we have
\[\sum_{m\in M}\chi(1+m)\:\phi (m)=0.\]
To that end, it suffices to show vanishing over each coset of the subgroup $M^\perp\subseteq M$. Namely, consider a sumlet of the form
\[\sum_{m\in M^\perp}\chi(1+m+r)\:\phi (m+r)\]
 where $r\in M$. For each $m\in M^\perp$, we can write $1+m+r=(1+r)(1+(1+r)^{-1}m)$ where $1+(1+r)^{-1}m\in 1+M^\perp$. But the multiplicative character $\chi$ is, by assumption, trivial on $1+M^\perp$. Therefore $\chi(1+m+r)=\chi(1+r)$ for each $m\in M^\perp$. Then
 \[\sum_{m\in M^\perp}\chi(1+m+r)\:\phi (m+r)=\chi(1+r)\: \phi(r) \sum_{m\in M^\perp}\phi (m)=0.\]
The argument is complete. \end{proof}

The computation of the third Kloosterman moment, in the case of a finite field $R$, is a classical result due to Sali\'e \cite{S} (see also \cite[Sec.4.4]{I}). The outcome is as follows: if $R$ is a field of odd characteristic, then
\begin{align*}
\sum_{a\in R^\times} K(a)^3&=\sigma(-3)\: |R|^2+2|R|+1.
\end{align*}
Note that the quadratic symbol of $-3$, explicit in the case of a field, features in the proof in the non-field case.

\section{Weighted moments}\label{sec: twistmom}
\begin{datum}
$R$ is a local Frobenius ring, and $\psi$ is a fixed primitive character of $R$.
\end{datum}

By \emph{weighted moments} of twisted Kloosterman sums we mean sums of the form
\begin{align*}
\sum_{a\in R^\times} \chi(a)\: K_\tau(a)^k, \quad \sum_{a\in R^\times} \chi(a)\: |K_\tau(a)|^k
\end{align*}
where $\chi$ is a multiplicative character of $R$, and $k=1,2,\dots$. Our goal in this section is the computation of the first two weighted moments. As we will see, the results are expressed in terms of Gauss and Jacobi sums.

\begin{thm}\label{thm: GKtwist}
Let $\tau$ be any twist. For each multiplicative character $\chi$ of $R$, the following holds:
\begin{align*}
\sum_{a\in R^\times} \chi(a)\: K_\tau(a)=G(\chi)\: G(\chi\tau).
\end{align*}
\end{thm}

\begin{proof} We have
\begin{align*}
\sum_{a\in R^\times} \chi(a)\: K_\tau(a)&=\sum_{a\in R^\times} \chi(a) \sum_{u\in R^\times} \tau(u)\:\psi(u+au^{-1})\\
&=\sum_{u\in R^\times} \tau(u)\: \psi(u) \sum_{a\in R^\times} \chi(a) \:\psi(au^{-1})\\
&=\sum_{u\in R^\times} \tau(u)\:\chi(u)\: \psi(u) \sum_{a\in R^\times} \chi(a) \:\psi(a)
\end{align*}
In the last step, we have made the change of variable $a:=au$ in the inner sum. Visibly, the outcome is $G(\chi\tau)\: G(\chi)$.
\end{proof}

In particular, taking $\chi=\ct$ we deduce the following.

\begin{cor}\label{cor: 1sttwist}
Assume that $R$ is not a field, and let $\tau$ be any twist. Then:
\[\sum_{a\in R^\times} K_\tau(a)=0.\]
\end{cor}

In effect, Theorem~\ref{thm: GKtwist} computes the multiplicative Fourier transform of the twisted Kloosterman function $K_\tau:R^\times\to \C$. With this interpretation in mind, we state the following.

\begin{lem}[Plancherel identity] Let $f:R^\times \to \C$ be a function. Then
\[\sum_{a\in R^\times} \big|f(a)\big|^2=\frac{1}{|R^\times|}\sum_{\chi\in \widehat{R^\times}} \:\Big|\sum_{a\in R^\times} \chi(a)\: f(a)\Big|^2.\]
\end{lem}

Our main use of Theorem~\ref{thm: GKtwist} is the following computation.

\begin{cor}\label{cor: 2ndtwist}
Assume that $R$ is not a field, and let $\tau$ be any twist. Then:
\begin{align*}\sum_{a\in R^\times} \big|K_\tau(a)\big|^2=\begin{cases} |R||R^\times|& \textrm{ if } \tau  \textrm{ is non-primitive,}\\
2|R||R^\times|-|R|^2 &  \textrm{ if } \tau  \textrm{ is primitive.}
\end{cases}
\end{align*}
\end{cor}

Note that, in the case when the twist $\tau$ is non-primitive, we can simply use Lemma~\ref{lem: link} and the full moment $\sum_{a\in R} \big|K_\tau(a)\big|^2=|R||R^\times|$ established in Proposition~\ref{prop: tK1}. However, we need a different approach in order to handle arbitrary twists.

\begin{proof}
By the Plancherel identity, we have
\begin{align*}\sum_{a\in R^\times} \big|K_\tau(a)\big|^2=\frac{1}{|R^\times|}\sum_{\chi\in \widehat{R^\times}} \:\Big|\sum_{a\in R^\times} \chi(a)\: K_\tau(a)\Big|^2=\frac{1}{|R^\times|}\sum_{\chi\in \widehat{R^\times}} \:\big|G(\chi)\: G(\chi\tau)\big|^2.
\end{align*}
We compute the right-hand side. Lemma~\ref{lem: gauss} gives 
\begin{align*}
\big|G(\chi)\big|=
\begin{cases} |R|^{1/2} & \textrm{ if } \chi  \textrm{ is primitive}\\
0 & \textrm{ otherwise}
\end{cases}
\end{align*}
and so
\[\frac{1}{|R^\times|}\sum_{\chi\in \widehat{R^\times}} \:\big|G(\chi)\: G(\chi\tau)\big|^2
=\frac{|R|^2}{|R^\times|} \:\big|\{\chi\in \widehat{R^\times}: \chi  \textrm{ and } \chi\tau  \textrm{ are primitive}\}\big|.\]

Consider first the case when the twist $\tau$ is non-primitive. Recall that, for a multiplicative character, to be non-primitive means to be trivial on the subgroup $1+M^\perp$. In this case the character subsets $\{\chi: \chi \textrm{ non-primitive}\}$ and $\{\chi: \chi\tau \textrm{ non-primitive}\}$ coincide. Therefore
\[\big|\{\chi\in \widehat{R^\times}: \chi  \textrm{ and } \chi\tau  \textrm{ are primitive}\}\big|=|R^\times|(1-|M^\perp|^{-1})\]
and we conclude that
\begin{align*}
\sum_{a\in R^\times} \big|K_\tau(a)\big|^2=|R|^2(1-|M^\perp|^{-1})=|R|\:(|R|-|M|)=|R||R^\times|.
\end{align*}

Consider now the case when the twist $\tau$ is primitive. Then the character subsets $\{\chi: \chi \textrm{ non-primitive}\}$ and $\{\chi: \chi\tau \textrm{ non-primitive}\}$ are disjoint, and they have the same cardinality. In fact, the former is a subgroup, while the latter is a non-trivial coset. In this case we have
\[\big|\{\chi\in \widehat{R^\times}: \chi  \textrm{ and } \chi\tau  \textrm{ are primitive}\}\big|=|R^\times|(1-2|M^\perp|^{-1})\]
and so
\begin{align*}
\sum_{a\in R^\times} \big|K_\tau(a)\big|^2=|R|^2(1-2|M^\perp|^{-1})=|R|\:(|R|-2|M|)=2|R||R^\times|-|R|^2.
\end{align*}
\end{proof}

Next, we turn to the second weighted moment. We distinguish two cases, according to whether the weighting character $\chi$ is primitive or not. 

\begin{thm}\label{thm: liutwist}
Let $\tau$ be any twist. For each primitive multiplicative character $\chi$ of $R$, the following holds:
\begin{align*}
\sum_{a\in R^\times} \chi(a)\: \big|K_\tau(a)\big|^2=\tau(-1)\: J(\chi\tau,\chi\overline{\tau})\:G(\chi) ^2.
\end{align*}
\end{thm}

\begin{proof} For $a\in R^\times$, we write:
\begin{align*}
\big|K_\tau(a)\big|^2&=K_\tau(a)\: \overline{K_\tau(a)}=\tau(-1)\:K_\tau(a)\:K_{\overline{\tau}}(a)\\
&=\tau(-1)\:\sum_{u,v\in R^\times} \tau(u)\:\overline{\tau}(v)\: \psi(u+v)\: \psi(a(u^{-1}+v^{-1}))\\
&=\tau(-1)\:\sum_{c\in R}\:\sum_{\substack{u,v\in R^\times \\ u+v=c}} \tau(u)\:\overline{\tau}(v)\: \psi(c)\: \psi(ac(uv)^{-1})
\end{align*}
Therefore
\begin{align*}
\sum_{a\in R^\times}  \chi(a)\: \big|K_\tau(a)\big|^2&=\tau(-1)\:\sum_{c\in R}\:\sum_{\substack{u,v\in R^\times \\ u+v=c}} \tau(u)\:\overline{\tau}(v)\: \psi(c)\sum_{a\in R^\times}  \chi(a)\: \psi(ac(uv)^{-1})\\
&=\tau(-1)\:\sum_{c\in R}\:\sum_{\substack{u,v\in R^\times \\ u+v=c}} \tau(u)\:\overline{\tau}(v)\: \chi(uv)\:\psi(c)\sum_{a\in R^\times}  \chi(a)\: \psi(ac)
\end{align*}
by making the change of variable $a:=auv$ in the inner sum. The latter inner sum is the Gauss sum, $G(\chi, c.\psi)$, and the sum over $u$ and $v$ can be collected into a generalized Jacobi sum:
\begin{align*}
\sum_{a\in R^\times}  \chi(a)\: \big|K_\tau(a)\big|^2=\tau(-1)\:\sum_{c\in R} J_c(\chi\tau,\chi\overline{\tau})\: G(\chi, c.\psi)\:\psi(c).
\end{align*}
Recall now the vanishing property of Gauss sums. As $\chi$ is primitive, $G(\chi, c.\psi)=0$ unless $c.\psi$ is primitive as well; in other words, $c\in R^\times$. Moreover, for $c\in R^\times$ we can rescale as follows:
\[G(\chi, c.\psi)=\chi(c^{-1}) \: G(\chi, \psi), \qquad J_c(\chi\tau,\chi\overline{\tau})=\chi^2(c)\: J(\chi\tau,\chi\overline{\tau}).\]
Summarizing, we have 
\begin{align*}
\sum_{a\in R^\times}  \chi(a)\: \big|K_\tau(a)\big|^2&=\tau(-1)\:\sum_{c\in R^\times} J_c(\chi\tau,\chi\overline{\tau})\: G(\chi, c.\psi)\:\psi(c)\\
&=\tau(-1)\: \sum_{c\in R^\times}\: \chi^2(c)\: J(\chi\tau,\chi\overline{\tau})\: \chi(c^{-1}) \: G(\chi)\: \psi(c)\\
&=\tau(-1)\: G(\chi)^2\:J(\chi\tau,\chi\overline{\tau})
\end{align*}
as claimed.
\end{proof}

We handle the second case, when the weighting character $\chi$ is non-primitive, under the assumption that the twist $\tau$ is non-primitive, and $R$ has odd characteristic. We start with the following averaging lemma. 

\begin{lem}\label{lem: cosetsum}
Assume $R$ is not a field, and has odd characteristic, and let $\tau$ be a non-primitive twist. Then for each $a\in R^\times$ we have
 \[\frac{1}{|M^\perp|}\sum_{b\in 1+M^\perp}  \big|K_\tau(ab)\big|^2=|R|\:(1+\sigma(a)).\]
\end{lem}

\begin{proof} To start off, let $b\in R^\times$. We write
\[K_\tau(ab)=\sum_{u\in R^\times} \tau(u)\: \psi(u)\:\psi(abu^{-1})=\tau(a)\sum_{u\in R^\times} \tau(u)\: \psi(au)\:\psi(bu^{-1})\]
by making a change of variable $u:=au$. Next, we have
\begin{align*}
\big|K_\tau(ab)\big|^2&=K_\tau(ab)\: \overline{K_\tau(ab)}=\tau(-(ab)^{-1})\:K_\tau(ab)^2\\
&=\tau(-ab^{-1})\sum_{u,v\in R^\times} \tau(uv)\: \psi(a(u+v))\: \psi(b(u^{-1}+v^{-1})).
\end{align*}
As the twist $\tau$ is non-primitive, $\tau(-ab^{-1})=\tau(-a)$ whenever $b\in 1+M^\perp$. Furthermore,
\begin{align*}
\sum_{b\in 1+M^\perp}\psi(b(u^{-1}+v^{-1}))&=\psi(u^{-1}+v^{-1})\sum_{b\in M^\perp}\psi(b(u^{-1}+v^{-1}))\\
&=\psi(u^{-1}+v^{-1})\: |M^\perp|\: [u^{-1}+v^{-1}\in M]
\end{align*}
by Lemma~\ref{lem: primsum}, and using $(M^\perp)^\perp=M$. Note also that $[u^{-1}+v^{-1}\in M]=[u+v\in M]$. At this point, we deduce
\begin{align*}
\sum_{b\in 1+M^\perp}  \big|K_\tau(ab)\big|^2=\tau(-a) \: |M^\perp| \sum_{\substack{u,v\in R^\times\\u+v\in M}} \tau(uv)\: \psi(a(u+v)+(u^{-1}+v^{-1}))
\end{align*}
In the latter sum, we group terms according to their sum and product:
\begin{align*}
\sum_{\substack{u,v\in R^\times\\u+v\in M}} \tau(uv)\: \psi(a(u+v))\: \psi(u^{-1}+v^{-1})=\sum_{c\in M}  \sum_{d\in R^\times} \sum_{\substack{u,v\in R^\times\\u+v=c, uv=d}} \tau(d)\: \psi(ac+cd^{-1}).
\end{align*}
 Now, by Lemma~\ref{lem: solutions}, we know the following: for $c\in M$ and $d\in R^\times$, the system $u+v=c$, $uv=d$ has $1+\sigma(-d)$ unit solutions. So we can write the latter triple sum as
 \begin{align*}
\sum_{c\in M}  \sum_{d\in R^\times} (1+\sigma(-d))\: \tau(d)\: \psi(ac+cd^{-1})=\sum_{c\in M}  \sum_{d\in R^\times} (1+\sigma(d))\: \tau(-d^{-1})\: \psi(c(a-d))
\end{align*}
after a change of variable $d:=-d^{-1}$. Furthermore,
\[\sum_{c\in M} \psi(c(a-d))=|M|\:[a-d\in M^\perp]\]
by Lemma~\ref{lem: primsum} once again. Summarizing our computation so far, we have
\begin{align*}
 \frac{1}{|M^\perp|}\sum_{b\in 1+M^\perp} \big|K_\tau(ab)\big|^2&=\tau(-a)\: |M| \sum_{d\in a+M^\perp} (1+\sigma(d))\: \tau(-d^{-1}).
 \end{align*}
 The inner sum has constant terms: both $\sigma$ and $\tau$ are trivial on $1+M^\perp$, so $\sigma(d)=\sigma(a)$ and $\tau(-d^{-1})=\tau(-a^{-1})$ whenever $d\in a+M^\perp$. Therefore
\begin{align*}
\frac{1}{|M^\perp|}\sum_{b\in 1+M^\perp} \big|K_\tau(ab)\big|^2&=\tau(-a)\: |M| |M^\perp|\:(1+\sigma(a))\: \tau(-a^{-1})=|R|\:(1+\sigma(a))
\end{align*}
 as desired.
\end{proof}

It is not hard to give a new proof of Theorem~\ref{thm: tK2}(ii) by using Lemma~\ref{lem: cosetsum}. We leave this as a puzzle for the reader. Our current goal is the following second weighted moment computation.

\begin{thm}\label{thm: liutwist2}
Assume $R$ is not a field, and has odd characteristic, and let $\tau$ be a non-primitive twist. Let $\chi$ be a non-primitive multiplicative character of $R$. Then
\begin{align*}\sum_{a\in R^\times} \chi(a)\: \big|K_\tau(a)\big|^2=\begin{cases} |R^\times||R| & \textrm{ if } \chi=\ct \textrm{ or } \chi=\sigma, \\ 0 & \textrm{ otherwise}.\end{cases}\end{align*}
\end{thm}

\begin{proof} The case when $\chi$ is trivial is known from Corollary~\ref{cor: 2ndtwist}, so let us assume that $\chi$ is non-trivial in what follows. As $\chi$ is non-primitive, there is a multiplicative character $\chi'$ on the quotient ring $R/M^\perp$ which induces $\chi$; furthermore, $\chi'$ is non-trivial. Using Lemma~\ref{lem: cosetsum}, we have:
\begin{align*}
\sum_{a\in R^\times} \chi(a)\: \big|K_\tau(a)\big|^2&=\sum_{v\in (R/M^\perp)^\times} \chi'(v) \sum_{b\in R^\times,\: [b]=v} \big|K_\tau(b)\big|^2\\
&=|R||M^\perp| \sum_{v\in (R/M^\perp)^\times} \chi'(v) (1+\sigma'(v))
\end{align*}
where $\sigma'$ is the quadratic character on $R/M^\perp$. Note that $\sum_{v\in (R/M^\perp)^\times} \chi'(v)=0$. Moreover, $\sum_{v\in (R/M^\perp)^\times}\: (\chi'\sigma')(v)=0$ as well--and so the second weighted moment vanishes--except when $\chi'=\sigma'$ on $R/M^\perp$. This amounts to $\chi=\sigma$, and in this case we get
\[\sum_{a\in R^\times} \chi(a)\: \big|K_\tau(a)\big|^2=|R||M^\perp||(R/M^\perp)^\times|=|R||R^\times|.\]
This is, of course, in complete agreement with the case $\chi=\ct$, given the vanishing on squares (Theorem~\ref{thm: tK2}(ii)).
\end{proof}

To end, we deduce the following variations on the above second weighted moment computations.

\begin{cor}
Let $\tau$ be a non-primitive twist. 

\begin{itemize}
\item[(i)] Let $\chi$ be a primitive multiplicative character of $R$. Then:
\begin{align*}
\sum_{a\in R^\times} \chi(a)\: K_\tau(a)^2=J(\chi,\chi\tau^2)\:G(\chi\tau) ^2.
\end{align*}
\item[(ii)] Let $\chi$ be a non-primitive multiplicative character of $R$. If $R$ is not a field, and has odd characteristic, then:
\begin{align*}\sum_{a\in R^\times} \chi(a)\: K_\tau(a)^2=\begin{cases} \tau(-1)\:|R^\times||R| & \textrm{ if } \chi=\overline{\tau} \textrm{ or } \chi=\sigma\overline{\tau}, \\ 0 & \textrm{ otherwise}.\end{cases}\end{align*}
\end{itemize}
\end{cor}

\begin{proof} For each $a\in R^\times$ we have
\begin{align*}
K_\tau(a)^2=\tau(-a)\:\overline{K_\tau(a)}\:K_\tau(a)=\tau(-a)\:\big|K_\tau(a)\big|^2
\end{align*}
and so
\begin{align*}
\sum_{a\in R^\times} \chi(a)\: K_\tau(a)^2=\tau(-1)\sum_{a\in R^\times} (\chi\tau)(a) \:\big|K_\tau(a)\big|^2.
\end{align*}

(i) If $\chi$ is primitive and $\tau$ is non-primitive, then $\chi\tau$ is primitive. Applying Theorem~\ref{thm: liutwist} on the right-hand sum leads to the desired result.

(ii) If $\chi$ is non-primitive and $\tau$ is non-primitive, then $\chi\tau$ is non-primitive as well. Now apply Theorem~\ref{thm: liutwist2} on the right-hand sum.
\end{proof}

\section{The fourth moment}\label{sec: 4mom}
\begin{datum}
$R$ is a local Frobenius ring, and $\psi$ is a fixed primitive character of $R$.
\end{datum}

Our interest in Theorems~\ref{thm: liutwist} and ~\ref{thm: liutwist2} is largely motivated by the observation that they offer a pathway to the computation of the fourth Kloosterman moment, in the case of non-primitive twists. Namely, the following holds.

\begin{thm}\label{thm:4mom}
Assume $R$ is not a field, and has odd characteristic, and let $\tau$ be a non-primitive twist. Then
\begin{align*}
\sum_{a\in R^\times} \big|K_\tau(a)\big|^4&=3|R^\times||R|^2.
\end{align*}
\end{thm}

\begin{proof}
By the Plancherel identity, we have
\[\sum_{a\in R^\times} \big|K_\tau(a)\big|^4=\sum_{a\in R^\times} \Big|\big|K_\tau(a)\big|^2\Big|^2=\frac{1}{|R^\times|}\sum_{\chi\in \widehat{R^\times}} \:\Big|\sum_{a\in R^\times} \chi(a)\: \big|K_\tau(a)\big|^2\Big|^2.\]
For each primitive multiplicative character $\chi$, we have
\[\Big|\sum_{a\in R^\times} \chi(a)\: \big|K_\tau(a)\big|^2\Big|=\big|J(\chi\tau,\chi\overline{\tau})\big|\:\big|G(\chi)\big| ^2\] 
by Theorem~\ref{thm: liutwist}. Recall from Lemma~\ref{lem: gauss}(ii) that $|G(\chi)|=|R|^{1/2}$ since $\chi$ is primitive. By Lemma~\ref{lem: jacobi}(iii),
\[G(\chi\tau)\: G(\chi\overline{\tau})=J(\chi\tau,\chi\overline{\tau})\:G(\chi^2)\]
provided that $\chi^2$ is primitive. Let us show that this is indeed the case. Arguing by contradiction, assume $\chi^2$ is non-primitive. This means that $\chi^2=\ct$ on the subgroup $1+M^\perp$. But the size of $1+M^\perp$ is odd, and so $(1+M^\perp)^2=1+M^\perp$. It follows that $\chi=\ct$ on $1+M^\perp$, which contradicts the primitivity of $\chi$.

As $\chi$ is primitive and $\tau$ is non-primitive, both $\chi\tau$ and $\chi\overline{\tau}$ are primitive as well. Thus $|G(\chi\tau)|=|G(\chi\overline{\tau})|=|G(\chi^2)|=|R|^{1/2}$, and we deduce that $|J(\chi\tau,\chi\overline{\tau})|=|R|^{1/2}$. In summary, for each primitive multiplicative character $\chi$ we have
\[\Big|\sum_{a\in R^\times} \chi(a)\: \big|K_\tau(a)\big|^2\Big|=|R|^{3/2}.\] 
Thus
\begin{align*}
\sum_{\substack{\chi\in \widehat{R^\times}\\ \chi \textrm{ primitive}}} \:\Big|\sum_{a\in R^\times} \chi(a)\: \big|K_\tau(a)\big|^2\Big|^2&=|R|^{3}\:\big|\{\chi\in \widehat{R^\times}\: : \: \chi \textrm{ primitive}\}\big|\\
&=|R|^3|R^\times|(1-|M^\perp|^{-1})=|R|^2\:|R^\times|^2. 
\end{align*}
On the other hand, the contribution from non-primitive multiplicative characters is easily computed thanks to Theorem~\ref{thm: liutwist2}:
\begin{align*}
\sum_{\substack{\chi\in \widehat{R^\times}\\ \chi \textrm{ non-primitive}}} \:\Big|\sum_{a\in R^\times} \chi(a)\: \big|K_\tau(a)\big|^2\Big|^2&=2|R|^2\:|R^\times|^2.
\end{align*}
We conclude that
\[\sum_{a\in R^\times} \big|K_\tau(a)\big|^4=\frac{1}{|R^\times|}(|R|^2\:|R^\times|^2+2|R|^2\:|R^\times|^2)=3 |R|^2|R^\times|\]
as desired.
\end{proof}


\end{document}